\documentclass[12pt]{amsart}

\usepackage{amsmath,amssymb,amscd,xcolor,amsthm,graphicx,enumerate}

\usepackage{hyperref}
\hypersetup{breaklinks=true}

\newcommand{\Balpha}{\mbox{$\hspace{0.1em}\rule[0.01em]{0.05em}{0.39em}\hspace{-0.21em}\alpha$}}

\numberwithin{equation}{section}
\theoremstyle{plain}

\makeatletter
\newtheorem*{rep@theorem}{\rep@title}
\newcommand{\newreptheorem}[2]{%
\newenvironment{rep#1}[1]{%
 \def\rep@title{#2 \ref{##1}}%
 \begin{rep@theorem}}%
 {\end{rep@theorem}}}
\makeatother

\newtheorem{theorem}[equation]{Theorem}
\newreptheorem{theorem}{Theorem}

\newtheorem{proposition}[equation]{Proposition}
\newtheorem{lemma}[equation]{Lemma}
\newtheorem{corollary}[equation]{Corollary}

\newtheorem{claim}[equation]{Claim}

\theoremstyle{remark}
\newtheorem{remark}[equation]{Remark}

\theoremstyle{definition}
\newtheorem{definition}[equation]{Definition}

\newtheorem{convention}[equation]{Convention}

\newtheorem*{question*}{Question}

\newcommand{\K}{{\mathcal K}}

\newcommand{\R}{\mathbb R}

\newcommand{\diam}{\operatorname{diam}}

\newcommand{\al}{\alpha}
\newcommand{\be}{\beta}
\newcommand{\D}{\partial}
\newcommand{\de}{\delta}

\newcommand{\Lap}{\Delta}

\newcommand{\eps}{\varepsilon}

\newcommand{\la}{\lambda}

\newcommand{\ol}{\overline}

\newcommand{\ra}{\rightarrow}

\newcommand{\spt}{\operatorname{spt}}

\providecommand{\abs}[1]{\lvert #1\rvert}

\def\XXint#1#2#3{{\setbox0=\hbox{$#1{#2#3}{\int}$}
     \vcenter{\hbox{$#2#3$}}\kern-.5\wd0}}

\begin{document}

\title[Mean curvature flow with surgery]{Mean curvature flow with surgery}
\author{Robert Haslhofer \and Bruce Kleiner}

\date{\today \thanks{B.K. was supported by NSF grant DMS-1105656.}}
\maketitle

\begin{abstract}
We give a new proof for the existence of mean curvature flow with surgery of $2$-convex hypersurfaces in $\R^N$, as announced in \cite{HK}.
Our proof works for all $N\geq 3$, including mean convex surfaces in $\mathbb{R}^3$.
We also derive a priori estimates for a more general class of flows in a local and flexible setting.
\end{abstract}

\tableofcontents

\section{Introduction}

\subsection{Background}
Mean curvature flow, like Ricci flow and harmonic map heat flow, develops singularities, and understanding them is a central problem in geometric analysis.
In the last decades several methods have been developed to contiunue the flow beyond the first singular time:
Brakke flow \cite{Brakke}, level set flow \cite{evans-spruck,CGG}, and mean curvature flow with surgery \cite{huisken-sinestrari3}; these have different advantages and are to some extent complementary.
Mean curvature flow with surgery gives more refined control on the topology of the evolving manifold, and can also be used to give smooth approximations to weak solutions \cite{Head,Lauer};
moreover, it has the nice feature that the entire discussion takes place in the framework of smooth differential geometry.

Roughly speaking, the idea of surgery is to continue the flow through singularities
by cutting along necks, gluing in caps, and continuing the flow of the pieces; components of known geometry and topology are discarded.
Huisken and Sinestrari successfully implemented this idea for the mean curvature flow of 2-convex hypersurfaces in $\R^N$ ($N\geq 4$),
i.e. for hypersurfaces where the sum of the smallest two principal curvatures is positive.
As in the construction of 3d Ricci flow with surgery by Perelman \cite{perelman_entropy,perelman_surgery} (see also \cite{Hamilton_pic,KL,MT,CaoZhu,BBBMP}),
this required considerable technical virtuosity, and was the culmination of a series of long papers \cite{huisken-sinestrari1,huisken-sinestrari2,huisken-sinestrari3}.

\subsection{Overview}
The purpose of the present paper is to give a new proof for the existence of mean curvature flow with surgery for $2$-convex hypersurfaces in $\R^N$ ($N\geq 3$), as announced in \cite{HK}.
The main focus is on shortness, brevity and simplicity.
Moreover, our proof works for mean convex surfaces in $\R^3$, which was left as an open problem after \cite{huisken-sinestrari3}, and recently solved also by Brendle-Huisken \cite{BH}, see Remark \ref{remark_BH}.
Furthermore, anticipating future applications, we derive our a priori estimates in a completely local and very flexible setting.

Our new approach to surgery is motivated by our paper \cite{HK}, where we gave a streamlined treatment of White's theory of mean convex level set (Brakke) flow \cite{white_size,white_nature,white_subsequent}, 
based on the beautiful noncollapsing result of Andrews \cite{andrews1}:
Given any $\alpha>0$, the condition that each boundary point of a mean convex domain admits interior and exterior balls of radius $\alpha/H(p)$ is preserved under the flow.

We formulate the existence theory for a class of flows that we call $(\Balpha,\delta,\mathbb{H})$-flows.
The parameters $\Balpha=(\alpha,\beta,\gamma)$ measure certain geometric quantities of the initial domain ($\alpha$-Andrews, $\lambda_1+\lambda_2\geq \beta H$, $H\leq\gamma$),
the parameter $\delta$ measures the quality of the surgery-necks, and the three curvature-scales $\mathbb{H}=(H_{\textrm{trig}},H_{\textrm{neck}},H_{\textrm{th}})$ are used to specify
more precisely when and how surgeries are performed, see Definition \ref{def_MCF_surgery}. 

Our main existence result, Theorem \ref{thm_main_existence}, proves the existence of an $(\Balpha,\delta,\mathbb{H})$-flow for any $2$-convex initial domain with parameters $\Balpha$.
It is complemented by a canonical neighborhood theorem, Theorem \ref{thm_can_nbd}, which gives a precise geometric description of the regions of high curvature.
Other geometric, topological and analytic conclusions follow immediately, see Corollary \ref{cor_discarded}, Corollary \ref{cor_topo}, and  Proposition \ref{cor_levelset}. 

The key for our existence proof are new a priori estimates for mean curvature flow with surgery. We formulate our a priori estimates for a more general class of flows that we call $(\alpha,\delta)$-flows.
They are concatenations of smooth $\alpha$-Andrews flows, where at finitely many times: (1) some surgeries are performed on $\delta$-necks of comparable scales; and/or (2) some connected components are discarded, see Definiton \ref{def_alphadelta}.

Our three main a priori estimates for $(\alpha,\delta)$-flows are a local curvature estimate (Theorem \ref{loccurvest}), a convexity estimate (Theorem \ref{thm-intro_convexity_estimate}),
and a global curvature estimate (Theorem \ref{thm_glob_curv_est}).
These estimates generalize to $(\alpha,\delta)$-flows our main estimates for $\alpha$-Andrews flows from \cite{HK}.
The presence of surgeries makes the proof of the local curvature estimate significantly more delicate and requires many new ideas, see Section \ref{sec_apriori}.
In the proof of the convexity estimate and the global curvature estimate, we then follow closely the scheme from our previous paper \cite{HK},
highlighting the new steps necessitated by surgeries.

Broadly speaking, once the a-priori estimates are established, the canonical neighborhood theorem follows from a short argument by contradiction, see Section \ref{subsec_cannbd}, and the main existence result follows from a relatively short continuity argument, see Section \ref{subsec_existence}
(strictly speaking, we also need certain structural results for ancient solutions and the evolution of standard caps, but their proofs are easy; see Section \ref{sec_ancientstandard}).

Finally, we point out that the notion of $(\al,\de)$-flows is considerably more flexible than the one of $(\Balpha,\delta,\mathbb{H})$-flows.
To a certain extent this additional flexibility is needed for locality, see Remark \ref{remark_whim}.
It also makes the class of flows larger, and thus the a priori estimates stronger.

\subsection{A priori estimates} We will now discuss our a priori estimates for $(\al,\de)$-flows. We start by recalling the simpler notion of $\al$-Andrews flows, which was the framework for our estimates in \cite{HK}.

\begin{definition}[{c.f. \cite{andrews1} and \cite[Def. 1.1]{HK}}]\label{def_alpha_andrews}
Let $\alpha>0$. A \emph{smooth $\alpha$-Andrews flow} $\{K_t\subseteq U\}_{t\in I}$ in an open set $U\subseteq \R^N$ over a time interval $I\subseteq\R$ is a smooth family of mean convex domains moving by mean curvature flow,
such that for every $p\in\D K_t$ the two closed balls $\bar{B}_{\textrm{int}}$ and $\bar{B}_{\textrm{ext}}$ that are tangent at $p$ and have radius $\frac{\alpha}{H(p)}$
satisfy $\bar{B}_{\textrm{int}}\cap U\subseteq K_t$ and $\bar{B}_{\textrm{ext}}\cap U\subseteq U\setminus\textrm{Int}(K_t)$, respectively.
\end{definition}

\begin{convention}
We now fix a factor $\mu\in [1,\infty)$, quantifying the notion of surgeries at comparable scales, for the class of all flows under consideration. Note however, that  we do \emph{not} fix any absolute scale.
\end{convention}

\begin{definition}\label{def_alphadelta}
An \emph{$(\alpha,\delta)$-flow} $\K$ is a collection of finitely many smooth $\al$-Andrews flows $\{K_t^i\subseteq U\}_{t\in[t_{i-1},t_{i}]}$ ($i=1,\ldots,k$; $t_0<\ldots< t_k$) in an open set $U\subseteq \R^N$ (see Definition \ref{def_alpha_andrews}),
such that:
\begin{enumerate}
\item for each $i=1,\ldots,k-1$, the final time slices of some collection of disjoint strong $\delta$-necks\footnote{More precisely, we only require that their final time-slices are disjoint in $U$.}
 are replaced by pairs of standard caps as described in Definition \ref{def_surgery},
 giving a domain $K^\sharp_{t_{i}}\subseteq K^{i}_{t_{i}}=:K^-_{t_{i}}$.
\item the initial time slice of the next flow, $K^{i+1}_{t_{i}}=:K^+_{t_{i}}$, is obtained from $K^\sharp_{t_{i}}$ by discarding some connected components.
\item there exists $s_\sharp=s_\sharp(\K)>0$, which depends on $\K$, such that all necks in item (1) have radius $s\in[\mu^{-1/2}s_\sharp,\mu^{1/2} s_\sharp]$.
\end{enumerate}
\end{definition}

\begin{remark} To avoid confusion, we emphasize that the word `some' allows for the empty set,
i.e. some of the inclusions $K_{t_i}^+\subseteq K_{t_i}^\sharp\subseteq K_{t_i}^-$ could actually be equalities. In other words, there can be some times $t_i$ where effectively only one of the steps (1) or (2) is carried out.
Also, the flow can become extinct, i.e. we allow the possibility that $K^{i+1}_{t_{i}}=\emptyset$.
\end{remark}

\begin{remark}\label{remark_whim}
At first sight, the possibility that some connected components can be discarded whenever someone wants might seem arbitrary, but in fact this flexibility is crucial for obtaining local estimates.
For example, imagine the situation that $\{L_t\subset \R^N\}_{t\in I}$ is a mean curvature flow with surgery in $\R^N$ (technically an $(\Balpha,\de,\mathbb{H})$-flow), and that we want to prove local estimates in an open set $U\subset \R^N$, i.e. estimates for the $(\al,\de)$-flow $\{K_t=L_t\cap U\subseteq U\}_{t\in I'}$, on some interval $I'\subseteq I$. It can of course happen that there is a surgery in $L_t$ on a neck outside $U$ that disconnects $L_t$ into two connected components. According to the standard procedure, the component of high curvature, which has known geometry and topology, is then discarded.
Viewed from $U$ one cannot see the neck and thus one observes a component being discarded seemingly on a whim.
\end{remark}

We will now state our three main a priori estimates for $(\al,\de)$-flows. We say that an $(\alpha,\de)$-flow is defined in a parabolic ball $P(p,t,r)=B(p,r)\times (t-r^2,t]$
 if it is defined in an open set $U\supseteq B(p,r)$ and for all $t'\in(t-r^2,t]$. Our first main estimate gives curvature control on a whole parabolic
ball, from a mean curvature bound at a single point.

\begin{theorem}[Local curvature estimate]\label{loccurvest}
There exist $\bar{\delta}=\bar{\delta}(\al)>0$, $\rho=\rho(\al)>0$ and $C_\ell=C_\ell(\alpha)<\infty$ ($\ell=0,1,2,\ldots$) with the following property.
If $\K$ is an $(\al,\delta)$-flow ($\delta\leq\bar{\de}$) in a parabolic ball $P(p,t,r)$ centered at a boundary point 
$p\in \D K_t$ with $H(p,t)\leq r^{-1}$, then
\begin{equation}\label{eqn_loccurv}
 \sup_{P(p,t,\rho r)\cap \partial \K}\abs{\nabla^\ell A}\leq C_\ell r^{-1-\ell}.
\end{equation}
\end{theorem}

Second, we have the convexity estimate for $(\al,\de)$-flows.

\begin{theorem}[Convexity estimate]\label{thm-intro_convexity_estimate}
For all $\eps>0$, there exist $\bar{\delta}=\bar{\delta}(\alpha)>0$ and $\eta=\eta(\eps,\al)<\infty$ with the following property.
If $\K$ is an $(\al,\delta)$-flow ($\delta\leq\bar{\delta}$) defined in a parabolic ball  $P(p,t,\eta\, r)$ centered at a boundary point 
$p\in \D K_t$ with $H(p,t)\leq r^{-1}$, then
\begin{equation}
\la_1(p,t)\geq -\eps r^{-1}.
\end{equation}  
\end{theorem}

Theorem \ref{thm-intro_convexity_estimate} says that a boundary point $(p,t)$ in an $(\al,\de)$-flow
has almost positive definite second fundamental form as long as the flow has had a chance to evolve over a portion of space-time which is large
compared with the scale given by $H(p,t)$. In particular, it implies that ancient $(\al,\de)$-flows in $\R^N$ have nonnegative second fundamental form.

Third, we have the global curvature estimate for $(\al,\de)$-flows.

\begin{theorem}[Global curvature estimate]\label{thm_glob_curv_est}
For all $\Lambda<\infty$,
there exist $\bar{\delta}=\bar{\delta}(\alpha)>0$, $\eta=\eta(\alpha,\Lambda)<\infty$ and $C_\ell=C_\ell(\alpha,\Lambda)<\infty$ ($\ell=0,1,2,\ldots$) with the following property.
If $\K$ is an $(\al,\delta)$-flow ($\delta\leq\bar{\delta}$) in a parabolic ball  $P(p,t,\eta\, r)$ centered at a boundary point 
$p\in \D K_t$ with $H(p,t)\leq r^{-1}$, then
\begin{equation}\label{eqn_globcurv}
 \sup_{P(p,t,\Lambda r)\cap\partial \K'}\abs{\nabla^\ell A}\leq C_\ell r^{-1-\ell},
\end{equation}
where $\K'$ denotes the $(\al,\de)$-flow in $P(p,t,\Lambda r)$ that is obtained from $\K$ via restricting to $B(p,\Lambda r)$ and discarding the connected components that do not contain $p$.
\end{theorem}

\begin{remark} Theorem \ref{thm_glob_curv_est} quickly implies (and immediately follows from) a global convergence result, Corollary \ref{thm_global_convergence}.
\end{remark}

\begin{remark}
Only normalizing the curvature at the basepoint one cannot get curvature control for other connected components.
An interesting feature, that shows up later in the proof of the canonical neighborhood theorem for $(\Balpha,\de,\mathbb{H})$-flows (Section \ref{subsec_cannbd}), is that the Andrews condition in combination with the degeneration of the surgery parameters provides a mechanism to clear out other components.
\end{remark}

\begin{remark}\label{remark_generic}
Our a priori estimates are local and thus also apply if mean convexity and the Andrews condition only hold locally around a singularity.
This seems to be very useful if one wants to implement surgery for generic mean curvature flow, c.f. \cite{CM_generic,CIM,CM_uniqueness}.
\end{remark}

\subsection{Existence theory}

We now turn to the discussion of the existence theory.
Our construction only depends on a few parameters of the initial domain, which is always assumed to be $2$-convex.

\begin{definition}[Controlled initial condition]\label{def_initialdata}
Let $\Balpha=(\al,\beta,\gamma)\in(0,N-2)\times (0,\tfrac{1}{N-2})\times (0,\infty)$.
A smooth compact domain $K_0\subset \R^N$ is called an \emph{$\Balpha$-controlled initial condition},
if it satisfies the $\alpha$-Andrews condition and the inequalities $\lambda_1+\lambda_2\geq \beta H$ and $H\leq \gamma$.
\end{definition}

\begin{remark}\label{remark_controlled}
Note that every smooth compact $2$-convex domain is a controlled initial condition for some parameters $\alpha,\beta,\gamma>0$.
\end{remark}

For us, a mean curvature flow with surgery is a $(\al,\de)$-flow with $(\al,\beta,\gamma)$-controlled initial data, subject to the following additional conditions.
First, the flow is \emph{$\beta$-uniformly $2$-convex}, i.e. $\lambda_1+\lambda_2\geq \beta H$.
Besides the neck-quality $\delta>0$, we have three curvature-scales $H_{\textrm{trig}}> H_{\textrm{neck}} > H_{\textrm{th}} > 1$, to which we refer as the trigger-, neck- and thick-curvature. 
The surgeries are done at times $t$ when the maximum of the mean curvature hits $H_{\textrm{trig}}$.
They are performed on a minimal disjoint collection of solid $\de$-necks of curvature $H_{\textrm{neck}}$ that separate the trigger part $\{H=H_{\textrm{trig}}\}$ from the thick part $\{H\leq H_{\textrm{th}}\}$ in $K_t^-$, and the high curvature components are discarded.
Finally, we impose the condition that surgeries are done more and more precisely, if the surgery-necks happen to be rounder and rounder.
We call our flows with surgery $(\Balpha,\de,\mathbb{H})$-flows, and the precise definition is as follows.

\begin{definition}\label{def_MCF_surgery}
An \emph{$(\Balpha,\de,\mathbb{H})$-flow}, $\mathbb{H}=(H_{\textrm{th}},H_{\textrm{neck}},H_{\textrm{trig}})$, is an $(\al,\de)$-flow $\{K_t\subset \R^N\}_{t\geq 0}$ (see Definition \ref{def_alphadelta}) with $\lambda_1+\lambda_2\geq \beta H$, and
with $\Balpha=(\al,\beta,\gamma)$-controlled initial condition $K_0\subset \R^N$ (see Definition \ref{def_initialdata}) such that
\begin{enumerate}
\item $H\leq H_{\textrm{trig}}$ everywhere, and
surgery and/or discarding occurs precisely at times $t$ when $H=H_{\textrm{trig}}$ somewhere.
\item The collection of necks in item (1) of Definition \ref{def_alphadelta} is a minimal collection of solid $\de$-necks of curvature $H_{\textrm{neck}}$ which
separate the set $\{H=H_{\textrm{trig}}\}$ from $\{H\leq H_{\textrm{th}}\}$ in the domain
$K_t^-$.
\item $K_t^+$ is obtained from $K_t^\sharp$ by discarding precisely those connected components with $H>H_{\textrm{th}}$ everywhere.
In particular, of each pair of facing surgery caps precisely one is discarded. 
\item If a strong $\delta$-neck from item (2) also is a strong $\hat{\delta}$-neck for some $\hat{\delta}<\delta$, then property (4) of Definition \ref{def_surgery} also holds with $\hat{\delta}$ instead of $\delta$.
\end{enumerate}
\end{definition}

\begin{remark}
To avoid confusion, we emphasize again that the collection of necks in item (2) can be empty. Indeed, the simplest example of an \emph{$(\Balpha,\de,\mathbb{H})$-flow} is the evolution of a round ball $K_0=\bar{B}^N$, which just shrinks self-similarly until $H$ reaches $H_{\textrm{trig}}$, and then is discarded.
\end{remark}

\begin{remark}\label{remark_extinct}
By comparison every $(\Balpha,\de,\mathbb{H})$-flow becomes extinct after some finite time $T$, i.e. satisfies $K_t=\emptyset$ for all $t>T$.
\end{remark}

\begin{remark}
We point out that our definition of  $(\Balpha,\de,\mathbb{H})$-flows does not include a canonical neighborhood property. Instead, we will prove that the canonical neighborhood property is consequence of the global convergence result (Corollary \ref{thm_global_convergence}), see Theorem \ref{thm_can_nbd}.
\end{remark}

Our main existence theorem is the following.

\begin{theorem}[Existence of mean curvature flow with surgery]\label{thm_main_existence}
There are constants $\ol{\de}=\ol{\de}(\Balpha)>0$ and $\Theta(\de)=\Theta(\Balpha,\de)<\infty$ ($\delta\leq\bar{\de}$) with the following significance.
If $\de\leq\bar{\de}$ and $\mathbb{H}=(H_{\textrm{trig}},H_{\textrm{neck}},H_{\textrm{th}})$ are positive numbers with
${H_{\textrm{trig}}}/{H_{\textrm{neck}}},{H_{\textrm{neck}}}/{H_{\textrm{th}}},H_{\textrm{neck}}\geq \Theta(\de)$,
then there exists an $(\Balpha,\de,\mathbb{H})$-flow $\{K_t\}_{t\in[0,\infty)}$ for every $\Balpha$-controlled initial condition $K_0$.  
\end{theorem}

Theorem \ref{thm_main_existence} enables us to evolve any smooth compact $2$-convex domain by mean curvature flow with surgery.
More generally, it also enables us evolve (finite- or infinite-dimensional) families of $\Balpha$-controlled initial domains simultaneously, with uniform estimates.

Our existence theorem (Theorem \ref{thm_main_existence}) is complemented by the following canonical neighborhood theorem (we use the terminology of Perelman \cite{perelman_surgery}, but our canonical neighborhoods are different), which gives precise geometric information about the regions of high curvature.

\begin{theorem}[Canonical neighborhood theorem]\label{thm_can_nbd}
For all $\eps>0$, there exist $\ol{\de}=\ol{\de}(\Balpha)>0$, $H_{\textrm{can}}(\eps)=H_{\textrm{can}}(\Balpha,\eps)<\infty$ and $\Theta_\eps(\delta)=\Theta_\eps(\Balpha,\delta)<\infty$ ($\delta\leq\bar{\de}$) with the following significance.
If $\de\leq\ol{\de}$ and $\K$ is an $(\Balpha,\de,\mathbb{H})$-flow with ${H_{\textrm{trig}}}/{H_{\textrm{neck}}},{H_{\textrm{neck}}}/{H_{\textrm{th}}}\geq \Theta_\eps(\de)$,
then any $(p,t)\in\D \K$ with $H(p,t)\geq H_{\textrm{can}}(\eps)$ is $\eps$-close to either
(a) a $\beta$-uniformly $2$-convex ancient $\al$-Andrews flow,
or (b) the evolution of a standard cap preceded by the evolution of a round cylinder.
\end{theorem}

\begin{remark}
 The structure of uniformly $2$-convex ancient $\al$-Andrews flows and the standard solution are discussed in Section \ref{sec_ancientstandard}.
\end{remark}

\begin{remark}\label{remark_practice}
In practice, one can first choose $\eps>0$ small enough, and then 
${H_{\textrm{trig}}}/{H_{\textrm{neck}}},{H_{\textrm{neck}}}/{H_{\textrm{th}}},H_{\textrm{neck}}\geq \max\{ \Theta_\eps(\de),\Theta(\delta)\}$ ($\delta\leq\bar{\delta}$)
and $H_\textrm{th}\geq H_{\textrm{can}}(\eps)$, where $\bar{\delta}=\min\{\bar{\delta}_{\ref{thm_main_existence}},\bar{\delta}_{\ref{thm_can_nbd}}\}$, so that both the existence result (Theorem \ref{thm_main_existence}) and the canonical neighborhood property at any point with $H(p,t)\geq H_{\textrm{can}}(\eps)$ (Theorem \ref{thm_can_nbd}) are applicable.
\end{remark}

In particular, $(\Balpha,\de,\mathbb{H})$-flows imply a decomposition of the initial domains into balls $\bar{D}^N$ and solid tori $\bar{D}^{N-1}\times S^1$.

\begin{corollary}[Discarded components]\label{cor_discarded}
For $\eps>0$ small enough, for any $(\Balpha,\de,\mathbb{H})$-flow with ${H_{\textrm{trig}}}/{H_{\textrm{neck}}},{H_{\textrm{neck}}}/{H_{\textrm{th}}}\geq \Theta_\eps(\de)$ ($\de\leq\bar{\de}$) and
$H_\textrm{th}\geq H_{\textrm{can}}(\eps)$, where $\Theta_\eps(\de)$, $\bar{\de}$ and $H_{\textrm{can}}(\eps)$ are from Theorem \ref{thm_can_nbd}, all discarded components are diffeomorphic to $\bar{D}^N$ or $\bar{D}^{N-1}\times S^1$.
\end{corollary}

\begin{corollary}\label{cor_topo}
Any smooth compact $2$-convex domain in $\R^N$ is diffeomorphic to a connected sum of finitely many solid tori $\bar{D}^{N-1}\times S^1$.
\end{corollary}

Finally, by the work of Lauer \cite{Lauer} we see that sequences of flows with surgery converge to the level set flow \cite{evans-spruck,CGG}, if the initial domain is kept fixed and the thick curvatures tend to infinity.

\begin{proposition}[Convergence to level set flow]\label{cor_levelset}
There exists $\bar{\delta}>0$ such that if $\K^j=\{K^j_t\}_{t\in[0,\infty)}$ is a sequence of 
$(\Balpha,\de_j,\mathbb{H}_j)$-flows ($\de_j\leq\bar{\de}$) starting at a fixed initial domain $K_0$,  with $H^j_{\textrm{th}}\to\infty$, 
then $\K^j$ Hausdorff converges in $\R^N\times [0,\infty)$ to $\K$, the level set flow of $K_0$.
\end{proposition}

Proposition \ref{cor_levelset} enables us to approximate level set flows with $\Balpha$-controlled initial conditions by $(\Balpha,\de,\mathbb{H})$-flows.
In fact, this approximation works uniformly for families of $\Balpha$-controlled initial conditions.

\begin{remark}\label{remark_BH}
We have announced our new construction of mean curvature flow with surgery in April 2013, see \cite{HK}.
In September 2013, Brendle and Huisken posted a very interesting preprint \cite{BH}, where they also solved the case $N=3$, 
by combining the work of Huisken-Sinestrari \cite{huisken-sinestrari3},
the local curvature estimate from our previous paper \cite{HK},
and Brendle's estimate for the inscribed radius \cite{Brendle_inscribed}.
Although Brendle's estimate is not needed at all for our approach to surgery, we have written a note \cite{HK_inscribed} giving a shorter proof of it.
Also, Brendle's interesting variant of the monotonicity formula \cite{Brendle_monotonicity} is somewhat related to Claim \ref{claim_monotonicity} in the present paper.
\end{remark}

\noindent\emph{Organization of the paper.} In Section \ref{sec_apriori}, we prove our a priori estimates for $(\al,\de)$-flows. 
In Section \ref{sec_ancientstandard}, we prove some results for ancient solutions and standard solutions, that are needed later.
In Section \ref{section_existence}, we prove the canonical neighborhood theorem and the main existence theorem.

\section{A priori estimates for flows with surgeries}\label{sec_apriori}

Our a priori estimates hold for flows in any open $U\subseteq \R^N$. 
Working locally necessitates a number of extra technicalities in the definitions and proofs; readers focussing on the global case can assume $U=\R^N$.  

\subsection{Basic properties of $(\al,\de)$-flows}\label{sec_basic}

We will now discuss some basic properties of $(\al,\de)$-flows (Definition \ref{def_alphadelta}).
We start by defining what it means to replace a neck by caps, making explicit the constants that are involved; this notion is used in item (1) of Definition \ref{def_alphadelta}.

\begin{convention}\label{conv_constants}
In the following, $\mu<\infty$, $C_\ell<\infty$ ($\ell=0,1,2,\ldots$) and $\delta'(\delta)$
with $\delta'$ decreasing to zero as $\delta\searrow 0$, are fixed but arbitrary.
Strictly speaking, the specification of $\mu$, $C_\ell$ and $\delta'(\delta)$ is part of the definition of an $(\alpha,\delta)$-flow. Other constants may depend on $\mu$, $C_\ell$ and $\delta'(\delta)$, but for brevity we suppress that in the notation.
Also, $\Gamma\in[10,\infty)$ denotes a constant to be fixed later (Convention \ref{convention_gamma}).
\end{convention}

\begin{definition}[Standard cap]\label{def_stdcap}
A \emph{standard cap} is a smooth convex domain $K^{\textrm{st}}\subset \R^N$ that coincides with a solid round half-cylinder of radius $1$ outside a ball of radius $10$.
\end{definition}

\begin{definition}[Strong $\delta$-neck]
We say that an $(\al,\de)$-flow $\K=\{K_t\subseteq U\}_{t\in I}$ has a \emph{strong $\delta$-neck} with center $p$ and radius $s$ at time $t_0\in I$, if
$\{s^{-1}\cdot(K_{t_0+s^2t}-p)\}_{t\in(-1,0]}$ is $\delta$-close in $C^{\lfloor 1/\delta\rfloor}$ in $B_{1/\delta}^U\times (-1,0]$ to the evolution of a solid round cylinder $\bar{D}^{N-1}\times \R$ with radius $1$ at $t=0$,
where $B_{1/\delta}^U=s^{-1}\cdot\left((B(p,s/\delta)\cap U)-p\right)\subseteq B(0,1/\delta)\subset \R^N$.
\end{definition}

\begin{definition}[Replacing a $\delta$-neck by standard caps]\label{def_surgery}
We say that the final time slice of a strong $\delta$-neck ($\delta\leq\tfrac{1}{10\Gamma}$) with center $p$ and radius $s$ is \emph{replaced by a pair of standard caps},
if the pre-surgery domain $K^-\subseteq U$ is replaced by a post-surgery domain $K^\sharp\subseteq K^-$ such that:
\begin{enumerate}
\item the modification takes places inside a ball $B=B(p,5\Gamma s)$.
 \item there are bounds for the second fundamental form and its derivatives:
$$\sup_{\D K^\sharp\cap B}\abs{\nabla^\ell A}\leq C_\ell s^{-1-\ell}\qquad (\ell=0,1,2,\ldots).$$
 \item if $B\subseteq U$, then for every point $p_\sharp\in \partial K^\sharp\cap B$ with $\lambda_1(p_\sharp)< 0$, there is a point
 $p_{-}\in\partial K^{-}\cap B$ with $\frac{\lambda_1}{H}(p_{-})\leq\frac{\lambda_1}{H}(p_{\sharp})$.
 \item if $B(p,10\Gamma s)\subseteq U$, then $s^{-1}\cdot(K^\sharp-p)$ is $\delta'(\delta)$-close in $B(0,10\Gamma)$ to a pair of disjoint standard caps,
that are at distance $\Gamma$ from the origin (see Convention \ref{conv_constants} and Definition \ref{def_stdcap}).
\end{enumerate}
\end{definition}

\begin{proposition}[Spatial separation of surgeries]\label{lemma_sepsurg}
For every $\zeta<\infty$, there are constants $\bar{\de}=\bar{\de}(\zeta)>0$ and $\eta=\eta(\zeta)<\infty$ with the following property.
If $\K$ is an $(\al,\de)$-flow ($\delta\leq\bar{\de}$), $p$ is the center of a surgery neck of radius $s$, and $\K$ is defined in an open set that contains the ball $B(p,\eta s)$, then are no other surgeries in $B(p,\zeta s)$.
\end{proposition}

\begin{proof}
Let $p'$ be the center of another surgery neck.
If the surgery occurs at the same time $t'=t$, then by disjointness of the collection of $\delta$-necks in $B(p,\eta s)$, we get $d(p,p')\geq \min\{\delta^{-1},\eta\}s$.
If $t'>t$, then by monotonicity of sets we have $K_{t'}\subseteq K_{t}$, and we also know that $K_{t}$ is close to a pair of standard caps near $p$.
Since a $\delta$-neck can only be contained in a standard cap when the center is far away from the tip, we obtain that $d(p,p')\geq\tfrac{1}{2}\mu^{-1} \min\{\delta^{-1},\eta\}s$ for $\eta$ large and $\delta$ small enough.
Finally, if $t'<t$, we can apply the same argument with the roles of $p$ and $p'$ reversed. The assertion follows.
\end{proof}

\begin{definition}[Points modified by surgery]\label{def_modifiedpoints}
We say that an open set $B$ contains \emph{points modified by surgery} at time $t$ if $(K_t^-\setminus K_t^\sharp)\cap B\neq \emptyset$.
\end{definition}

\begin{proposition}[Forward estimate after surgery]\label{lemma_pseudo}
For all $\al>0$, there exist $\eps=\eps(\al)>0$, $C=C(\al)<\infty$ and $\bar{\delta}=\bar{\delta}(\alpha)>0$ with the following property.
If $\K$ as an $(\alpha,\delta)$-flow ($\delta\leq\bar{\delta}$),
and $r\leq s$ and $t_1<t_2$ are such the flow is defined in $B(p,r)\times [t_1,t_2]$, with some point in $B(p,r)$ modified by a surgery
at scale $s$ at time $t_1$, but no points in $B(p,r)$ modified by surgeries for $t\in(t_1,t_2]$, 
then
\begin{equation}
\sup_{\D K_t\cap B(p,r/2)}\abs{A}\leq Cr^{-1} \quad\textrm{on}\quad [t_1,\min\{t_1+(\eps r)^2,t_2\}].
\end{equation}
\end{proposition}

\begin{proof}
By Definition \ref{def_surgery} and the Andrews-condition we have uniformly controlled geometry in $B(p,r)$ at time $t_1$ (the curvature bounds follow from item (2) for $\partial K^\sharp_{t_1}\cap B(p,r)$, and from being part of a strong $\delta$-neck at other points).
Thus, the assertion follows from the pseudolocality theorem for the mean curvature flow, see e.g. \cite[Thm. 7.5]{ChenYin}.\footnote{Note that dropping connected components has the good sign in Huisken's monotonicity inequality. Thus, their proof of pseudolocality goes through in our setting.}
\end{proof}

\begin{proposition}[One-sided minimization]\label{lemma_onesidedmin}
There exist $\bar{\delta}>0$ and $\Gamma_0<\infty$ with the following property.
If $\K$ is an $(\alpha,\delta)$-flow ($\delta\leq\bar{\delta}$) in an open set $U$, with cap separation parameter $\Gamma\geq \Gamma_0$ and surgeries at scales between $\mu^{-1}s$ and $s$,
and if $\bar{B}\subset U$ is a closed ball with $d(\bar{B},\R^N\setminus U)\geq 20\Gamma s$,
then
\begin{equation}
\abs{\partial K_{t_1}\cap \bar{B}}\leq\abs{\partial K'\cap \bar{B}}
\end{equation}
for every smooth comparison domain $K'$ that agrees with $K_{t_1}$ outside $\bar{B}$ and satisfies $K_{t_1}\subseteq K'\subseteq K_{t_0}$ for some $t_0<t_1$.
\end{proposition}

\begin{proof}[Proof]
This follows using the geometric measure theory argument in \cite[Sec. 3]{white_size} and \cite[Sec. 5]{Head}; in Appendix \ref{appendix_min} we give a detailed exposition, and also an alternative elementary argument.
\end{proof}

\begin{convention}\label{convention_gamma}
We now fix a constant $\Gamma\geq\max \{\Gamma_0,10\}$, where $\Gamma_0$ is from Proposition \ref{lemma_onesidedmin}, that is also large enough such that there are constants $\omega<1$ and $\bar{\delta}>0$ with
\begin{equation}\label{area_drop}
 \abs{\D K^+\cap B}\leq \omega \abs{\D K^-\cap B},
\end{equation}
for every surgery on a $\delta$-neck ($\delta\leq\bar{\delta}$) with $B=B(p,5\Gamma s)\subseteq U$.
\end{convention}

\subsection{Local curvature estimate}\label{loc_curv}
We will now prove Theorem \ref{loccurvest}. Since the proof is somewhat involved, we first give a detailed outline.

\noindent\emph{Outline of the proof.} Arguing by contradiction, we get a sequence of flows on larger and larger parabolic balls where the curvature goes to zero at the basepoint but blows up at some nearby point.
We first split off the two easy cases that there are no nearby surgeries (Case 1) or surgeries at macroscopic scales (Case 2), which can be dealt with by applying the local curvature estimate from our previous paper \cite[Thm. 1.8]{HK} and the forward estimate after surgeries (Proposition \ref{lemma_pseudo}), respectively.
The core of the proof is then to rule out surgeries at microscopic scales (Case 3).
We do this as follows: At a surgery neck the value of the Huisken density is close to the value of the cylinder.
However, since the mean curvature at the basepoint goes to zero, using a halfspace convergence argument (Claim \ref{claim_halfspace}) and one-sided minimization (Proposition \ref{lemma_onesidedmin}), we can show that the Huisken density is close to $1$ further back in time.
Finally, analyzing the contributions from surgeries in different regimes, c.f. \eqref{eq_regime}, we prove that the cumulative error in Huisken's monotonicity inequality due to surgeries goes to zero (Claim \ref{claim_monotonicity}), and conclude that microscopic surgeries cannot occur.

\begin{proof}[Proof of Theorem \ref{loccurvest}] As in \cite[Sec. 2]{HK}, we will first prove \eqref{eqn_loccurv} under the admissibility assumption that some time slice $K_{\bar t}$ contains $B(p,r)$.
 Assume towards a contradiction that there exists a sequence $\K^j$ of admissible $(\al,j^{-1})$-flows in $P(0,0,j)$ with $H(0,0)\leq j^{-1}$, but such that
\begin{equation}\label{loccurv_contr}
\sup_{ P(0,0,j^{-2})}\abs{A}\geq j^2.
\end{equation}

\noindent\emph{Case 1: there are no points modified by surgeries in $P(0,0,j^{-1})$, for infinitely many $j$}. In this case, we can apply the local curvature estimate from our previous paper \cite[Thm. 1.8, Rmk. 2.12]{HK} (see also Remark \ref{alternative_proof} below for an alternative proof). For large $j$ this gives a contradiction with (\ref{loccurv_contr}).

\noindent\emph{Case 2: there is a point in $P(0,0,j^{-1})$ modified by a surgery at scale $s_j\geq \mu \eps^{-1}j^{-1}$, for infinitely many $j$}, where $\eps$ is the constant from Proposition \ref{lemma_pseudo}.
Let $t_1^j$ be the largest $t\in (-j^{-2},0]$ such that there is a point in $B(0,\eps^{-1}j^{-1})$ modified by a surgery.
Using $j^{-1}$-closeness to a strong neck we get curvature estimates prior to $t_1^j$, and using Proposition \ref{lemma_pseudo} we get curvature estimates after $t_1^j$.
For large $j$ this gives a contradiction with (\ref{loccurv_contr}).

\noindent\emph{Case 3: there is a point in $P(0,0,j^{-1})$ modified by a surgery, and all surgeries are at scale $s_j\leq \mu^2\eps^{-1}j^{-1}$, for large $j$.} We will show that this case actually cannot occur.
Suppressing $j$ in the notation, let $x_0$ be the center of neck coming from a surgery in $P(0,0,j^{-1})$,
at time $t$ and scale $s$. Let $t_0=t+\frac{s^2}{2(N-2)}$, and consider the Huisken density
\begin{equation}
 \Theta(\tau)=\int_{\partial { K}_{t_0-\tau}} \theta^\sigma_{X_0} dA,
\end{equation}
based at $X_0=(x_0,t_0)$. Here, the integrand $\theta^\sigma_{X_0}$ is the backwards heat kernel times a suitable cutoff function at scale $\sigma\gg 1$ (as always, we tacitly assume that $j$ is large enough), namely
\begin{equation}
 \theta^\sigma_{X_0}(x,t)=(4\pi(t_0-t))^{-(N-1)/2} e^{-\frac{\abs{x-x_0}^2}{4(t_0-t)}}(1-\tfrac{\abs{x-x_0}^2+2(N-1)(t-t_0)}{\sigma^2})_+^3.
\end{equation}
Note that for small backwards time, say $\tau=s^2$, we have
\begin{equation}\label{lowerdensity}
\liminf_{j\to\infty}\Theta(s^2)\geq \Theta_{N-1}>1,
\end{equation}
where $\Theta_{N-1}$ is the density of the cylinder $\bar{D}^{N-1}\times\R$.
Recall that by Huisken's monotonicity inequality, see e.g. \cite[Prop. 4.17]{Ecker_book} or \cite[App. B]{HK}, the function $\Theta(\tau)$ is monotone if there are no surgeries (throwing away connected components has the good sign).
\begin{claim}\label{claim_monotonicity}
 The cumulative error in Huisken's monotonicity inequality due to surgeries between $\tau=s^2$ and $\tau=1$ goes to zero as $j\to\infty$.
\end{claim}
\begin{proof}
Let $\eps>0$, and write $\sigma_t=\sqrt{\sigma^2-2(N-1)(t-t_0)}$. We say that a surgery center $x_i$ at time $t_i$ is in the \emph{nonoscillating regime},
if 
\begin{equation}\label{eq_regime}
 \textrm{(a)}\,\,\, \tfrac{s\abs{x_i-x_0}}{\abs{t_i-t_0}}< \eps\qquad
\textrm{and} \qquad \textrm{(b)}\,\,\, \abs{\abs{x_i-x_0}-\sigma_{t_i}}> \eps^{-1}s.
\end{equation}
If $\eps$ is sufficiently small, then the change of the Huisken density due to any surgery $(x_i,t_i)$ in the nonoscillating regime has the good sign, i.e.
\begin{equation}
 \int_{\partial { K_{t_i}^\sharp}\cap B_i} \theta^\sigma_{X_0} dA\leq \int_{\partial { K_{t_i}^-}\cap B_i} \theta^\sigma_{X_0} dA.
\end{equation}
where $B_i=B(x_i,5\Gamma s_i)$. Indeed, this follows from the fact that area decreases by a definite factor under surgery, see \eqref{area_drop}, and the observation that we can make the ratio
between $\sup_{x\in B_i}\theta^\sigma_{X_0}(x,t_i)$ and $\inf_{x\in B_i}\theta^\sigma_{X_0}(x,t_i)$ as close to $1$ as we want, by choosing $\eps$ small enough.

We next estimate the cumulative error due to the surgeries $(x_i,t_i)$ violating (a). Taking into account Proposition \ref{lemma_sepsurg}, it suffices to estimate the sum
\begin{equation}\label{sumtoest}
 \sum_i\frac{1}{\tau_i^{(N-1)/2}}e^{-\abs{x_i-x_0}^2/5\tau_i}A_i,
\end{equation}
where $\tau_i=t_0-t_i$ and $A_i$ is the area of the region modified by the surgery.
Here, we used that $\abs{x-x_0}^2\geq\tfrac{4}{5}\abs{x_i-x_0}^2$ for $x$ in the region around $x_i$ modified by surgery, for $j$ large.
To estimate (\ref{sumtoest}), we first pull out a factor $e^{-\abs{x_i-x_0}^2/10\tau_i}$.
Using Proposition \ref{lemma_sepsurg} again, we observe that (for $j$ large enough) the minimum of $\frac{\abs{x_i-x_0}^2}{10\tau_i}+\frac{N-1}{2}\log \tau_i$ over $\tau_i$
under the constraint $\tau_i\leq \frac{s\abs{x_i-x_0}}{\eps}$ is attained at  $\tau_i= \frac{s\abs{x_i-x_0}}{\eps}$. Thus
\begin{equation}\label{eq_compsum}
 \sum_i\frac{1}{\tau_i^{(N-1)/2}}e^{-\frac{\abs{x_i-x_0}^2}{5\tau_i}}A_i\leq\delta_j
 \sum_i \left(\frac{\eps}{s\abs{x_i-x_0}}\right)^{(N-1)/2}e^{-\frac{\eps\abs{x_i-x_0}}{10s}}A_i,
\end{equation}
where $\delta_j:=\sup_i e^{-\abs{x_i-x_0}^2/10\tau_i}$ goes to zero as $j\to \infty$, since
\begin{equation}
\frac{\abs{x_i-x_0}^2}{\tau_i}=\frac{s\abs{x_i-x_0}}{\tau_i}\frac{\abs{x_i-x_0}}{s}\geq \eps\zeta_j\to\infty,
\end{equation}
again by Proposition \ref{lemma_sepsurg}.
Since the regions modified by surgeries are separated by a large multiple of $s$ by Proposition \ref{lemma_sepsurg}, and have area comparable to $s^{N-1}$ by Definition \ref{def_surgery}, the sum on the right hand side of \eqref{eq_compsum} can be uniformly estimated by a multiple of 
\begin{equation}
\int_{s}^\infty \frac{1}{(s R)^{(N-1)/2}}e^{-R/s}R^{N-2}dR=\int_{1}^\infty u^{(N-3)/2}e^{-u}du<\infty.
\end{equation}
Thus, the cumulative error due to surgeries violating (a) goes to zero as $j\to\infty$.

Finally, observing that $\theta^\sigma_{X_0}(x,t)\leq s^3$ in the relevant region, the cumulative error due to the surgeries violating (b) can be estimated by a multiple of
$s^3(\tfrac{\sigma}{s})^Ns^{N-1}\to 0$ ($j\to\infty$). This proves the claim.
\end{proof}

\begin{claim}\label{claim_halfspace}
After rotating coordinates, the sequence $ \K^j$ converges in
the pointed Hausdorff topology on $\R^{N}\times(-\infty,0]$ to the static halfspace $\{x_N\leq 0\}\times(-\infty,0]$,
and similarly for the complements.
\end{claim}

\begin{proof}
We rotate coordinates such that the outward unit normal of $ K^j_0$ at $(0,0)$ is $e_N$.
Since $H(0,0)\leq j^{-1}$, the Andrews condition implies in particular that every compact subset of the lower half plane $\{x_N<0\}$ is contained in $K^j_0$ for $j$ large enough.
The proof is now almost verbatim as the one of \cite[Claim 2.4]{HK}, taking into account the following caveat. In our previous proof we used comparison from the interior with the mean curvature evolution of a ball $\bar{B}^d_R$.
In general, interior comparison can fail for $(\alpha,\delta)$-flows, since there can be jumps due to surgeries and thrown away components.
However, in our situation the radius of the comparison ball is much larger than the surgery scale $s_j\to 0$.
Thus, the comparison ball must be disjoint from any region modified by surgeries, since any such region is contained in a long and thin neck.
Since the comparison ball contains the origin, it cannot be contained in a discarded component either.
\end{proof}

Finishing the discussion of Case 3, it follows from the halfspace convergence (Claim \ref{claim_halfspace}) and the one-sided minimization (Proposition \ref{lemma_onesidedmin}) that at $\tau=1$ the value of the Huisken density is close to $1$; together with the almost monotonicity (Claim \ref{claim_monotonicity})  this contradicts (\ref{lowerdensity}).

Finally, the admissibility assumption can be removed as in \cite[App. D]{HK},\footnote{Readers focussing on the admissible version of the estimate can skip this step.} and replacing $\rho$ by $\rho/2$, the bounds for the derivatives of $A$ follow from standard interior estimates.
\end{proof}

\begin{remark}\label{alternative_proof}
A simple nonsurgical variant of the argument in Case 3 gives an alternative proof of \cite[Thm. 1.8, Rmk. 2.12]{HK}.
Namely, given a contradictory sequence $\K^j$ as at the beginning of the proof of Theorem \ref{loccurvest}, select points $X_j\in\partial\K^j\cap P(0,0,1/2)$ such that $Q_j=\abs{A}(X_j)\geq j^2$ and $\sup_{P(X_j,{j}Q_j^{-1})}\abs{A}\leq 2Q_j$.
Then, the Huisken density $\Theta_{X_j}$ based at $X_j$ satisfies $\liminf_{j\to\infty}\Theta_{X_j}(Q_j^{-2})>1$.
Arguing as in Case $3$, for $\tau=1$ we are as close as we want to a halfspace; together with the one-sided minimization, this contradicts monotonicity.
\end{remark}

\subsection{Convexity estimate} Using the local curvature estimate (Theorem \ref{loccurvest}), we can now prove the convexity estimate (Theorem \ref{thm-intro_convexity_estimate}).
The idea is, as in {\cite[{Proof of Thm. 1.8}]{HK}},
to consider a contradictory sequence and to pass to a local limit such that $\frac{\la_1}{H}$ attains a negative minimum, contradicting the strict maximum principle.
To make this idea work in the presence of surgeries, we have to choose the sequence very carefully, using in particular item (3) of Definition \ref{def_surgery}.

\begin{proof} Fix $\al $, and let $\ol{\delta}=\ol{\delta}(\alpha)>0$ small enough 
to justify the application of Theorem \ref{loccurvest} and of the properties of Definition \ref{def_surgery} in the argument below.
The $\al$-Andrews condition implies that for every $\eps \geq \frac1\al $, we can find an $\eta<\infty$ such that the assertion holds.  
Let $\eps_0\leq \frac1\al $ be the infimum of the $\eps$'s for which this is possible, 
and suppose $\eps_0>0$.

It follows that there is a sequence $\{\K^j\}$ of $(\al,\delta_j)$-flows, $\delta_j\leq\bar{\de}$, in $P(0,0,j)$ such that $H(0,0)\leq 1$, but ${\la_1}(0,0)\to -\eps_0$ as $j\ra \infty$.
By the choice of $\eps_0$, it follows that $H(0,0)\ra 1$ as $j\ra\infty$, since
otherwise we could parabolically scale our sequence and get a new sequence
where $H(0,0)\leq 1$, but $\la_1(0,0)$ tends to something strictly smaller than
$-\eps_0$.

Let $\rho=\rho(\al)$ be the quantity from the local curvature estimate
(Theorem \ref{loccurvest}).  Then there are uniform bounds on $A$ and its
space-time derivatives in $P(0,0,\rho/2)$.  Suppose there is no $r>0$ such that
the flow is unmodified by surgeries in $P(0,0,r)$ after passing to a subsequence.
In view of the bounds on spacetime derivatives of curvature, we may assume, after translating and parabolic rescaling (by a factors tending to
$1$ as $j\ra\infty$), that $t=0$ is a surgery time, and that $(0,0)$
lies in $\partial K_0^\sharp\cap B(p,5\Gamma s)$, c.f. Definition \ref{def_surgery}.
The radius of the surgery neck is comparable to one, again by Definition \ref{def_surgery}.
Thus, by item (3) of Definition \ref{def_surgery}, after passing to some point at controlled distance in
the presurgery manifold, and parabolically rescaling by factors of controlled size, we may assume that $(0,0)$ lies in the presurgery manifold $\partial K_0^-$, $H(0,0)=1$, and $\lambda_1(0,0)\to -\eps_0$. After modifying the sequence in this way, the argument can now be concluded as in {\cite[{Proof of Thm. 1.8}]{HK}}.
Namely, using Proposition \ref{lemma_sepsurg} and Theorem \ref{loccurvest} we get a smooth mean curvature flow $\K^\infty$ in some parabolic ball $P(0,0,r)$ such that the ratio $\frac{\la_1}{H}$ attains a 
negative minimum $-\eps_0$ at $(0,0)$; this contradicts the strict maximum principle.
\end{proof}

\subsection{Global curvature estimate}\label{subsec_global_conv}
The global convergence theorem from our previous paper (\cite[Thm. 1.12]{HK}) was based on the local curvature estimate and the convexity estimate.
Having established the local curvature estimate and the convexity estimate for $(\al,\de)$-flows (Theorem \ref{loccurvest} and Theorem \ref{thm-intro_convexity_estimate}),
we will now show that our previous global convergence argument goes through with minor adjustments.

\begin{proof}[Proof of Theorem {\ref{thm_glob_curv_est}}]
We choose $\bar{\delta}=\bar{\delta}(\alpha)>0$ small enough such that the estimates from the previous sections apply.
Suppose towards a contradiction, that there is a sequence $\K^j$ of $(\alpha,\delta_j)$-flows ($\delta_j\leq\bar{\de}$) in $P(0,0,\eta_j)$, with $\eta_j\to\infty$ and $H(0,0)\leq 1$, such that
\begin{equation}\label{eqn_contr_ass}
 \lim_{j\to\infty}\sup_{P(0,0,\Lambda)\cap \D{\K}'^{j}}{\abs{A}}=\infty,
\end{equation}
for some $\Lambda < \infty$, where ${\K}'^{j}$ denotes the $(\alpha,\delta_j)$-flow whose time slices are given by the connected component of $K_t\cap B(0,\Lambda)$ containing $0$.

We can assume that there is some $R<\infty$ such that $P(0,0,R)$ contains surgeries of $\K^j$ for large $j$, since
otherwise \cite[Thm. 1.12, Rmk. 3.5]{HK} gives a contradiction with \eqref{eqn_contr_ass}. Also, it must be the case that the surgery scales $s_j=s_\sharp(\K^j)$ (see Definition \ref{def_alphadelta}) satisfy
\begin{equation}\label{eqn_scalebound}
\limsup_{j\to\infty} s_j<\infty,
\end{equation}
since otherwise Proposition \ref{lemma_pseudo} forwards in time and the strong $\delta_j$-neck assumption backwards in time, c.f. Case 2 of the proof of Theorem \ref{loccurvest}, gives curvature bounds contradicting again \eqref{eqn_contr_ass}.

After these preliminary reductions and observations, the proof is now verbatim as in \cite[Proof of Thm. 1.12]{HK}, apart from
some obvious changes in wording, like replacing $\al$-Andrews flow by $(\al,\de)$-flow, and from three minor modifications which we will carefully discuss now.

\noindent\emph{Modification 1:}
Instead of the nonsurgical version of the local curvature estimate \cite[Thm 1.8]{HK} and the convexity estimate \cite[Thm 1.10]{HK} we of course use the versions for $(\alpha,\delta)$-flows established in the present paper, Theorem \ref{loccurvest} and Theorem \ref{thm-intro_convexity_estimate}, respectively.

\noindent\emph{Modification 2:} The final paragraph of \cite[Proof of Thm. 1.12, Step 2]{HK}
needs to be expanded, since the $(\al,\de)$-flow $\hat{\K}^\infty$ might contain surgeries in $P(q_1,0,r)$.
If some neighborhood of $q_1$ is unmodified by surgeries at $t=0$ (Definition \ref{def_modifiedpoints}), our previous argument applies.
Otherwise, recall that $\hat{\K}^\infty$ arises as smooth limit of $(\al,\de)$-flows $\hat{\K}^j$.
After passing to a subsequence, we may assume that the surgery scales $s_\sharp(\hat{\K}^j)$ converge to a limit $s$, which must be comparable to $H^{-1}(q_1,0)$, by Theorem \ref{loccurvest} and Definition \ref{def_surgery}. 
Let $q_1'$ be a point on the radial segment in the cone $X_1$ connecting $q_1$ and the tip, such that $H(q_1',0)\gg H(q_1,0)$.
Since all surgeries are done at comparable scales (Definition \ref{def_alphadelta}), Theorem \ref{loccurvest} implies that for some $r'>0$,
the intersection $X_1\cap B(q_1',r')$ can be extended to a smooth $(\al,\de)$-flow $\hat{\K}'^\infty$ without surgeries in $P(q_1',0,r')$, and our previous argument applies.  

\noindent\emph{Modification 3:}
In \cite[Proof of Thm. 1.12, Step 7]{HK} comparison with large enough spheres containing the origin is still legitimate,
thanks to \eqref{eqn_scalebound} and $0\in X^j_{R,t}$, c.f. the proof of Claim \ref{claim_halfspace}. Thus, as in \cite[(3.4)]{HK} we obtain the estimate
\begin{equation}
\label{eqn-universal_h_bound2}
\lim_{j\ra\infty}\left(\sup_{\D X^j_{R,t}}H  \right)\leq f(R,t)\, 
\end{equation}
for some continuous function $f$ (at surgery times the estimate holds both for the pre- and post-surgery domain); this contradicts \eqref{eqn_contr_ass}.

Finally, the curvature bounds for the derivatives of the second fundamental form follow from standard interior estimates. 
\end{proof}

As mentioned in the introduction, the global curvature estimate (Theorem \ref{thm_glob_curv_est}) enables us to smoothly pass to global limits. It can happen that the limit contains infinitely many surgeries, but we do get a bound for the number of surgeries contained in any compact set.

\begin{definition}[Generalized $(\al,\de)$-flow]\label{def_generalized}
 A \emph{generalized $(\al,\de)$-flow} is a family of closed sets in $\R^N$ that is an $(\al,\de)$-flow when restricted to any open set $U\subset \R^N$ with compact closure. 
\end{definition}

\begin{corollary}[Global convergence]\label{thm_global_convergence}
There exists $\bar{\de}=\bar{\de}(\alpha)>0$ with the following property. If $\K^j$ is a sequence of $(\alpha,\delta_j)$-flows, $\de_j\leq\bar{\de}$,
in $P(p_j,t_j,\eta_jH^{-1}(p_j,t_j))$ with $\eta_j\to\infty$, then, after passing to a subsequence,
the $(\al,\de)$-flows $\hat{\K}^j$ that are obtained from $\K^j$ by parabolic rescaling $(p,t)\mapsto (H(p_j,t_j)(p-p_j),H^{2}(p_j,t_j)(t-t_j))$, restricting to
$B(0,\Lambda_j)$ for a suitable sequence $\Lambda_j\to\infty$, and discarding the connected components that don't contain the origin,
converge smoothly and globally to a limit $\K^\infty=\{K_t^\infty\subset \R^N\}_{t\in(-\infty,0]}$,
which is a generalized $(\al,\de)$-flow with convex time slices.
\end{corollary}

In Corollary \ref{thm_global_convergence} the precise meaning of convergence is as follows.

\begin{definition}[Smooth convergence]\label{rem_smooth_convergence}
Let $\K^j$ be a sequence of $(\al,\de_j)$-flows, $\de_j\leq\bar{\de}$, with connected time slices, normalized such that $H(0,0)=1$, and defined in $P(0,0,\Lambda_j)$ with $\Lambda_j\to\infty$, and let $\K^\infty=\{K^\infty_t\subset\R^N\}_{t\in (-\infty,0]}$ be a generalized $(\al,\de)$-flow (Definition \ref{def_generalized}). We say that \emph{$\K^j$ converges to $\K^\infty$ smoothly and globally}, if $\K^j$ converges smoothly to $\K^\infty$ away from the regions modified by surgeries, and if near every point $p\in K^{\infty,-}_{t_\infty}$ that that is modified by a surgery at time $t_\infty$ the following condition is satisfied. There exists a sequence of surgery times $t_j\to t_\infty$ in $\K^j$ such that if we consider the forward and backward portions $\K^j_{+}=\{K^j_{t+t_j}-p\}_{t\geq 0^+}$ and $\K^j_{-}=\{K^j_{t+t_j}-p\}_{t\leq 0^-}$, and likewise $\K^\infty_{\pm}$, then $\K^j_\pm$ converges smoothly to $\K^\infty_\pm$ in a forward respectively backward parabolic neighborhood $P_{\pm}(0,0,\eps)$, for some $\eps>0$.
\end{definition}

\begin{proof}[Proof of Corollary \ref{thm_global_convergence}]
Let $\bar{\delta}=\bar{\delta}(\alpha)>0$ small enough such that the previous estimates apply.
Let $\K^j$ be a sequence of $(\alpha,\de_j)$-flows ($\de_j\leq\bar{\de}$) in $P(0,0,\eta_j)$ ($\eta_j\to\infty$) with $H(0,0)=1$.
Choose $\Lambda_j\to\infty$ slowly enough such that the conclusion of Theorem \ref{thm_glob_curv_est} holds for the flow $\K'^j$,
that is obtained from $\K^j$ by restricting to $B(0,\Lambda_j)$ and discarding the connected components that don't contain the origin.
We want to find a subsequence of $\K'^j$ that converges smoothly and globally.

We can assume that there is some $R<\infty$ such that $P(0,0,R)$ contains surgeries of $\K^j$ for large $j$, since
otherwise \cite[Thm. 1.12, Rmk. 3.5]{HK} allows us to pass to a smooth limit. Also, it must be the case that the surgery  scales $s_j=s_\sharp(\K^j)$ (see Definition \ref{def_alphadelta}) satisfy \eqref{eqn_scalebound} since otherwise Proposition \ref{lemma_pseudo} gives a contradiction with $H(0,0)=1$.

For each positive integer $k$, by Theorem \ref{thm_glob_curv_est}, inequality \eqref{eqn_scalebound}, and Proposition \ref{lemma_sepsurg},
the parabolic ball $P(0,0,k)$ contains at most some controlled number $N_k^j$ of surgeries of $\K'^j$, and their necks are of controlled size.
After passing to a subsequence, we can assume that $N_k^j$ equals some fixed number $N_k$, that the surgery times converge to some limiting surgery times given by a set $\mathcal{T}_k\subset (-\infty,0]$ with at most $N_k$ elements, and that the (pre and post) surgery time slices converge smoothly.
Let $\mathcal{Q}\subset (-\infty,0]$ be a countable dense set that is disjoint from $\cup_k\mathcal{T}_k$. Arguing as in \cite[Proof of Thm. 1.12, Step 7]{HK}, after passing to a subsequence there are convex sets $K^{\infty}_t$
such that the domains $X^{j}_{R,t}$ converge smoothly to 
$K^{\infty}_t\cap B(0,R)$ as $j\ra\infty$, for all $R<\infty$ and all $t\in\mathcal{Q}$.
Finally, putting everything together, namely the convergence at a dense set of times, the convergence at the surgery times, Theorem \ref{thm_glob_curv_est}, Proposition \ref{lemma_sepsurg} and the Andrews-condition, it follows that there exists a generalized $(\al,\de)$-flow $\K^\infty$ such that $\K'^j\to \K^\infty$ smoothly and globally, where the meaning of convergence is as in Definition \ref{rem_smooth_convergence}.  
\end{proof}

\section{Ancient solutions and standard solutions}\label{sec_ancientstandard}

We will now prove some structural results for uniformly $2$-convex ancient $\al$-Andrews flows, and results for the evolution of standard caps.
These results will be used later in the blowup analysis in Section \ref{section_existence}.

\subsection{Structure of uniformly $2$-convex ancient $\al$-Andrews flows}\label{sec_ancient}

In this section, we consider smooth ancient $\alpha$-Andrews flows $\{K_t\subset \R^N\}_{t\in (-\infty,T)}$, that are \emph{$\beta$-uniformly $2$-convex}, 
i.e. $\la_1+\la_2\geq \be H$ for some fixed $\be>0$.
Examples to keep in mind are the cylinder, the bowl soliton \cite{AltWu}, the sphere and the Angenent ovals \cite{white_nature,HH}.

We recall that ancient $\al$-Andrews flows are always convex \cite[Cor. 2.15]{HK}, and in fact automatically smooth until they become extinct \cite[Thm. 1.14]{HK}.
We will now discuss two structural results that are more specific to the uniformly $2$-convex case. First, non $\eps$-neck points are at
controlled distance from one another, unless the time slice is compact and
the points approximately realize the diameter.

\begin{proposition}[Non-neck points]\label{structure_compact}
For all $\varepsilon_1,\eps_2>0$ there exists $\underline{R}=\underline{R}(\eps_1,\eps_2,\alpha,\beta)<\infty$, such that if $\K$ is a $\beta$-uniformly 
$2$-convex ancient $\al$-Andrews flow and $p_1,p_2\in \D K_t$ are not strong $\eps_1$-neck points, then at least
one of the following holds:
\begin{enumerate}
\item $\max_iH(p_i)d(p_1,p_2)<\underline{R}$.
\item $\diam K_t\leq (1+\eps_2)d(p_1,p_2)$.
\end{enumerate}
\end{proposition}

Second, as suggested by the bowl soliton, non-compact non-cylindrical solutions have a single cylindrical end in a precise quantitative sense.

\begin{proposition}[Quantitative one-endedness]\label{lem_quantitative_one_ended}
For all $\eps_1,\eps_2>0$ there exists $\underline{R}=\underline{R}(\eps_1,\eps_2,\al,\beta)<\infty$, such 
that if $\K$ is a $\beta$-uniformly $2$-convex
non-compact ancient $\al$-Andrews flow and $p\in \D K_t$ is not a strong $\eps_1$-neck point, then for any $R> \underline{R}$, 
there exists a strong $\eps_2$-neck point $q\in \D K_t \cap  S(p,RH^{-1}(p,t))$ such that 
the intersection
$K_t\cap S(p,RH^{-1}(p,t))$ is contained in the $\eps_2 H^{-1}(q,t)$ neighborhood
of a cross-sectional disc of the solid $\eps_2$-neck at $q$.
\end{proposition}

For the proofs we need the following two lemmas.

\begin{lemma}\label{lemma_ancient2}
If $\K$ is a uniformly $2$-convex ancient $\al$-Andrews flow that becomes extinct at a finite time $T<\infty$, then the final
time slice $K_T$ is a convex set of dimension at most $1$.\footnote{Conjecturally, $K_T$ is either a point or the entire real line, c.f. \cite{white_nature}.} 
\end{lemma}
\begin{proof}[Proof of Lemma \ref{lemma_ancient2}]
This follows immediately by combining \cite[Thm. 1.14]{HK} and \cite[Refinement of Thm. 1.15]{HK}.\footnote{It is of course possible to phrase this argument entirely in the smooth setting.}
\end{proof}

\begin{lemma}\label{lemma_limitneck}
Suppose  $\{\K^j\}$ is a sequence of $\beta$-uniformly $2$-convex ancient $\al$-Andrews flows, 
for which the time-slices $K^j_0$ Hausdorff-converge to a $1$-dimensional convex set $K^\infty$ containing $p$ as an interior point.
Then for any $\de>0$, any sequence of points $p_j\in \D K^j_0$ with
$p_j\ra p$ consists of strong $\de$-neck points for large $j$.
\end{lemma}

\begin{remark}
 We point out that $K^\infty$ can be strictly smaller than the final time slice $K^\infty_0$ of the space-time Hausdorff limit $\K^j\to\K^\infty$.
For example, if $\K^j$ is a blowdown sequence of the bowl soliton centered at the tip $(0,0)$, then $K^\infty$ is a half-line, but $K^\infty_0$ is the whole line (since the limit $\K^\infty$ is a shrinking cylinder that becomes extinct at time $0$).
\end{remark}

\begin{proof}[Proof of Lemma \ref{lemma_limitneck}]
We have $H(p_j)\ra \infty$ as $j\ra \infty$, for otherwise $K^\infty$
would be a domain with nonempty interior. Choose $x^\pm\in K^\infty$ such that $p$ lies in the interior of the 
segment $\ol{x^-x^+}$, and choose $x_j^\pm\in K^j_{t_j}$
such that $x_j^\pm\ra x^\pm$. The segments $\ol{x_j^- p_j}$ and $\ol{p_j x_j^+}$ are contained in $K^j_{t_j}$ by convexity, and the angle between them converges to $\pi$.
Parabolically rescaling to normalize
$H(p_j)$ and passing to a limit \cite[Thm. 1.12]{HK}, we obtain an ancient uniformly $2$-convex $\al$-Andrews flow
$\K$ that contains a line. Using \cite[Lemma 3.14]{HK} it follows that $\K$ is a round cylindrical flow.
Thus $p_j$ is a strong $\de$-neck point for large $j$. 
\end{proof}

\begin{proof}[Proof of Proposition \ref{structure_compact}]
Assume towards a contradiction that there are a sequence $\{\K^j\}$ of $\beta$-uniformly $2$-convex ancient 
$\al$-Andrews flows, and sequences of points $p_1^j,p_2^j\in K_0^j$ which are not strong $\eps_1$-neck points,
such that 
$\max_iH(p^j_i)d(p^j_1,p^j_2)\ra \infty$,
and $\diam K_0>(1+\eps_2)d(p^j_1,p^j_2)$
for all $j$.  Rescaling so that $d(p^j_1,p^j_2)=1$ we get $\max_iH(p_i^j)\ra \infty$.
After passing to a subsequence, we may assume that $p^j_i\ra p^\infty_i$
for $i=1,2$, and that $\hat{\K}^j$ converges, by \cite[Thm 1.14]{HK}, to an ancient $\al$-Andrews flow
$\hat{\K}^\infty$ which goes extinct at $T=0$, where $\diam \hat{K}^\infty\geq 1+\eps_2$.
By Lemma \ref{lemma_ancient2}, at least one of the points $p^\infty_1,p^\infty_2$ must be an interior point. Thus, Lemma \ref{lemma_limitneck} gives a contradiction to the assumption that $\{p_1^j,p_2^j\}$ are not strong $\eps_1$-neck points.
\end{proof}

\begin{proof}[Proof of Proposition \ref{lem_quantitative_one_ended}]
Assume towards a contradiction that there are a sequence
$\{\K^j\}$ of $\beta$-uniformly $2$-convex non-compact ancient $\al$-Andrews flows, and a sequence
$R_j\ra \infty$, such that $(0,0)\in \D K^j_0$ is not a strong $\eps_1$-neck point, $H(0,0)=1$, 
and $\D K^j_0\cap S(0,R_j)$ does not contain any point with the asserted property.

By non-compactness and convexity there is a ray in $K^j_0$ starting at the origin. Let $x_j$ be the point where the ray intersects $S(0,R_j)$, and let $q_j$ be a point in 
$K^j_0\cap S(0,R_j)$ with maximal distance from $x_j$; thus we have $q_j\in \D K^j_0$.
We claim that 
$\frac{d(x_j,q_j)}{R_j}\ra 0$; indeed, this follows by 
rescaling by $R_j^{-1}$ and passing to a limit \cite[Thm 1.14]{HK}, and using the fact that the limit $K^\infty$ of the time zero slices is non-compact, convex, and $1$-dimensional (since it is contained in $K^\infty_0$, the time zero slice of the space-time limit $\K^\infty$, which is at most $1$-dimensional by Lemma \ref{lemma_ancient2}).
Using Lemma \ref{lemma_limitneck}, we see that $0$ must be an endpoint, and thus that $K^\infty$ is a ray starting at $0$.

After shifting $x_j$ to the origin,  
parabolically rescaling by $d(x_j,q_j)^{-1}$,
and passing to  a subsequence, we obtain a new sequence $\hat\K^j$ which 
strongly Hausdorff converges to an ancient $\al$-Andrews flow 
$\hat \K^\infty$. Note that $\hat K^\infty_0$ must actually be smooth, since a convex set of dimension $1$ cannot contain $2$ perpendicular segments. Since
$\hat\K^\infty$ contains a line and is uniformly $2$-convex, it must be a round cylindrical flow, c.f. the proof of Proposition \ref{structure_compact}.
Moreover, $(S(0,R_j),x_j)$ converges after rescaling by $d(x_j,q_j)^{-1}$ to a hyperplane orthogonal to the axis of the cylinder, and we get a contradiction.
\end{proof}

\begin{remark}\label{remark_one_ended_cpt}
The proof of Proposition \ref{lem_quantitative_one_ended} also works if the assumption that the flow is noncompact is replaced by the assumption that there exists a point in $K_t$ with distance from $p$ at least $(1+\eps)R$.
\end{remark}

For later use (namely for the proof of Claim \ref{claim_sepclaim}), we also prove that the ratio between intrinsic and extrinsic distance is controlled.

\begin{proposition}[Intrinsic distance]
\label{lem_quasi_convex}
For every $R<\infty$, there is an $L=L(R,\al)<\infty$ such that 
for every ancient  $\al$-Andrews flow $\K$ and every point $(p,t)\in\D K_t$, any point
$x\in \D K_t\cap B(p,RH^{-1}(p,t))$ can be joined to $p$ by a path in $\D K_t$
of length at most $LH^{-1}(p,t)$.
\end{proposition}

\begin{proof}
If not, there is a sequence $\K^j$ of ancient $\al$-Andrews flows with $H(0,0)=1$, and a sequence of points $x_j\in \D K_0^j\cap B(0,R)$ that cannot be joined to $0$ by a path in $\partial K_0^j$ of length at most $j$.
By \cite[Thm. 1.12]{HK} we can pass to a subsequential limit $\K^\infty$ with convex (in particular connected) time slices. Since $K_0^\infty$ cannot be a slab, the boundary $\D K_0^\infty$ is also connected, which gives a contradiction for $j$ large enough.
\end{proof}

\subsection{The standard surgery solution}
\label{sec_standard_solution}

We will first consider standard caps  $K^{\textrm{st}}$ as in Definition \ref{def_stdcap}, that are $\alpha$-Andrews and $\beta$-uniformly $2$-convex for some values $\alpha=\alpha(K^{\textrm{st}}), \beta=\beta(K^{\textrm{st}})>0$. Afterwards, we will construct a particular model of $K^{\textrm{st}}$ that is suitable for surgeries.

\begin{proposition}\label{prop_std_sol}
Let $K^{\textrm{st}}$ be a standard cap as in Definition \ref{def_stdcap}, with $\alpha=\alpha(K^{\textrm{st}}), \beta=\beta(K^{\textrm{st}})>0$. There is a unique mean curvature flow $\{K_t\}_{t\in [0,1/2(N-2))}$ starting at $K^{\textrm{st}}$. It has the following properties:
\begin{enumerate}
\item It is $\alpha$-Andrews, convex, and $\beta$-uniformly $2$-convex.
\item There are continuous increasing functions $\underline{H},\ol{H}:[0,\tfrac{1}{2(N-2)})\ra \R_+$
with $\underline{H}(t)\ra \infty$ as $t\ra\tfrac{1}{2(N-2)}$ such that $\underline{H}(t)\leq H(p,t)\leq \ol{H}(t)$ for all $p\in\D K_t$ and $t\in [0,\tfrac{1}{2(N-2)})$.
\item For every $\eps>0$ and $\tau<\tfrac{1}{2(N-2)}$ there exists an $R=R(\eps,\tau)<\infty$ such that
outside $B(0,R)$ the flow $\{K_t\}_{t\in [0,\tau]}$ is $\eps$-close to the evolution of the solid round unit cylinder $\bar{D}^{N-1}\times\R$. 
\item For every $\eps>0$, there exists a $\tau=\tau(\eps)<\tfrac{1}{2(N-2)}$ such that every
point $(p,t)\in \D K_t$ with $t\geq \tau$ is $\eps$-close to a $\beta$-uniformly $2$-convex
ancient $\al$-Andrews flow.
\end{enumerate}
\end{proposition}

\begin{proof}[Proof]
The argument is closely related to the one for Ricci flow \cite{perelman_surgery,KL,MT}, so we will give a brief, but complete, treatment.

Existence and uniqueness of a smooth solution on a maximal time interval $[0,T)$ with bounded curvature on compact subintervals follows from standard theory, see e.g. \cite[Thm. 4.2]{EckerHuisken}, \cite[Thm. 1.1]{ChenYin}.
Moreover, since $A=A^i_j$ evolves by $\partial_t A=\Lap A+\abs{A}^2 A$,
it follows from the tensor-maximum principle, see e.g. \cite[Thm. 12.34]{Ricciflow_book}, that convexity and $\beta$-uniform 2-convexity are preserved along the flow.

Assume (3) fails for some $\eps>0$, $\tau<T$. Since we have uniform curvature bounds on $[0,\tau]$, we can pass to smooth limits. Thus, as a pointed limit around a suitable sequence of points going to infinity we get a mean curvature flow with bounded curvature that starts from a round cylinder, but is not $\eps$-close to the standard evolution of the round cylinder; this contradicts uniqueness of the evolution of the cylinder.

Next, we recall from \cite[Thm. 2]{Andrews_Langford_McCoy} that the quantities $\underline{Z}$, $\overline{Z}$ introduced there, satisfy the evolution inequalities
\begin{equation}
 \partial_t \underline{Z}\geq \Lap \underline{Z}+\abs{A}^2\underline{Z},\qquad \partial_t \overline{Z}\leq \Lap \overline{Z}+\abs{A}^2\overline{Z}
\end{equation}
in the viscosity sense. Arguing as in \cite[Cor. 3]{Andrews_Langford_McCoy}, it follows that the $\al$-Andrews condition is preserved;\footnote{Since we already know that the solution is asymptotically cylindrical, there is no subtle part at all about localizing the maximum principle.} this completes the proof of (1).

Assume $T$ is strictly less than $\tfrac{1}{2(N-2)}$.
Then, since we have shown that (3) holds for $\tau<T$, by the local curvature estimate \cite[Thm. 1.8, Rem. 2.10]{HK}, we obtain uniform bounds on the curvature near infinity; hence the curvature has to blow up inside some compact set.
Select a sequence of points $(p_j,t_j)$ with $Q_j=H(p_j,t_j)\to\infty$ and $H\leq 2 Q_j$ on $P(p_j,t_j,jQ_j^{-1})$. By convexity and since the curvature stays bounded outside some compact set, the time slices $K_{t_j}$ contain some cones based at $p_j$ with a definite lower bound on the cone angle.
Rescaling by  $Q_j^{-1}$ and passing to a limit, we obtain an ancient $\alpha$-Andrews flow with nonvanishing asymptotic volume ratio; this contradicts \cite[Rem. 1.20]{HK}.
Thus $T=\tfrac{1}{2(N-2)}$.

Since the solution is contained in a shrinking cylinder that becomes extinct at $T=\tfrac{1}{2(N-2)}$, the Andrews condition implies that the curvature must blow up everywhere as $t\to\tfrac{1}{2(N-2)}$, i.e. we get (2).

Assertion (4) follows from (2) and the global convergence theorem \cite[Thm. 1.12]{HK}.

Suppose $\{K_t'\}_{t\in [0,T')}$ is any smooth mean curvature flow starting at $K^{\textrm{st}}$, without any a priori bound on curvature.  Let $t_1\leq \min\{T',\frac{1}{2(N-2)}\}$ be the supremum of
the numbers $t$ such that $\K'$ coincides with $\K$ constructed above.
If $t_1< \min\{T',\frac{1}{2(N-2)}\}$, then $\hat K_{t_1}=K_{t_1}$, and
by pseudolocality  \cite[Thm. 7.5]{ChenYin} the curvature of $\K'$ is bounded uniformly
for a time interval $[t_1,t_2)$ for some $t_2>t_1$, contradicting the
uniqueness theorem.  Thus we obtain a unique solution without the bounded curvature assumption.
\end{proof}

\begin{proposition}[Existence of a suitable standard cap]\label{lemma_glue_caps}
Given any $\Balpha=(\al,\beta,\gamma)\in(0,N-2)\times (0,\tfrac{1}{N-2})\times (0,\infty)$, there exist $K^{\textrm{st}}\subset \R^N$, $\al^{\textrm{st}}>\al$, $\beta^{\textrm{st}}>\beta$, and $\bar{\delta}>0$ (all depending on $\Balpha$) with the following properties:
\begin{enumerate}
\item $K^{\textrm{st}}$ is a standard cap (see Definition \ref{def_stdcap})
\item $K^{\textrm{st}}$ is $\al^{\textrm{st}}$-Andrews and $\beta^{\textrm{st}}$-uniformly $2$-convex.
\item If $\delta\leq\bar{\delta}$, then it is possible to do surgery such that all properties in Definition \ref{def_surgery} hold, and such that in addition the $\alpha$-Andrews condition and the $\beta$-uniform $2$-convexity are preserved.
\end{enumerate}
\end{proposition}

\begin{proof}
We can take $K^{\textrm{st}}$ a smooth convex domain that is a small perturbation of a solid half cylinder with a half ball attached. This obviously satisfies (1) and (2).
 Let $\{N_t\}$ be a strong $\delta$-neck with optimal quality $\delta$ small enough. Since the neck is strong by assumption, by standard interior estimates we get uniform bounds for all derivatives of the curvatures. It is then clear, that we can cut along the final time slice of the neck and glue in two copies of $K^{\textrm{st}}$ such that item (1), (2) and (4) of Definition \ref{def_surgery} hold, and such that the $\alpha$-Andrews condition and the $\beta$-uniform $2$-convexity are preserved also. The only somewhat non-obvious point is item (3) of  Definition \ref{def_surgery}, but as observed by Hamilton \cite{Hamilton_pic} (see also \cite[Sec. 72]{KL} or \cite[p. 155]{huisken-sinestrari3}) this can be ensured by bending the cylinder slightly inwards.
\end{proof}

\section{Existence of mean curvature flow with surgery}\label{section_existence}

We keep the parameters $\Balpha=(\alpha,\beta,\gamma)$ fixed for this section. 

\subsection{The canonical neighborhood theorem}\label{subsec_cannbd}

In the proof of our main existence result, Theorem \ref{thm_main_existence}, we will consider sequences of flows where the $\mathbb{H}$-parameters degenerate suitably.
We will now prove the following crucial self-improvement phenomenon: If the surgeries are done on necks, where we a priori only know that they have at least some small but fixed quality $\bar{\delta}$, then the degeneration of the other parameters actually forces them to be more and more precise. In fact, this is just a reformulation of the canonical neighborhood theorem.

\begin{theorem}[Self-improvement of necks]\label{prop_selfimprove}
There exists a constant $\bar{\delta}=\bar{\delta}(\Balpha)>0$ with the following property.
If $\K^j$ is a sequence of $(\Balpha,\de_j,\mathbb{H}_j)$-flows ($\delta_j\leq \bar{\delta}$) with ${H^j_{\textrm{trig}}}/{H^j_{\textrm{neck}}},{H^j_{\textrm{neck}}}/{H^j_{\textrm{th}}}\ra \infty$, 
and if $(p_j,t_j)\in \D \K^j$ is a sequence of points with $H(p_j,t_j)\ra \infty$, then
after parabolic rescaling to normalize $H$ at $(p_j,t_j)$ and passing to a subsequence we have smooth convergence to either (a) a $\beta$-uniformly $2$-convex ancient $\al$-Andrews flow, 
or (b) the evolution of a standard surgery cap preceded by the evolution of a round cylindrical flow.\footnote{In case (b), the meaning of smooth convergence is as in
Definition \ref{rem_smooth_convergence}; also, the notion of $\eps$-closeness in Theorem \ref{thm_can_nbd} should be interpreted accordingly.}  
\end{theorem}

\begin{remark} Note that Theorem \ref{prop_selfimprove} is equivalent to Theorem \ref{thm_can_nbd}.
\end{remark}

\noindent\emph{Outline of the proof.} The proof amounts to classifying the limits, whose existence is guaranteed by the global convergence result (Corollary \ref{thm_global_convergence}).
If the limit doesn't contain surgeries, then it must be a $\beta$-uniformly $2$-convex ancient $\alpha$-Andrews flow. If the limit contains a surgery, then we argue, using in particular part (2) of Definition \ref{def_MCF_surgery}, the assumption that the curvature ratios degenerate, and the global curvature estimate (Theorem \ref{thm_glob_curv_est}), that the limit must contain a line. It is then easy to conclude that there is in fact only one surgery, and that the limit must have the structure as claimed.
Finally, we observe that potential other connected components get cleared out.

\begin{proof}
Let $\bar{\delta}=\bar{\delta}(\Balpha)>0$ small enough such that the estimates from the previous sections apply.
Since the initial domain has principal curvatures bounded by $\gamma/\alpha$,
the curvature remains bounded for a definite amount of time. Thus, the rescaled flows $\hat{\K}^j$ are defined on parabolic balls $P(0,0,\eta_j)$ with $\eta_j\ra\infty$ as $j\ra\infty$.
By Corollary \ref{thm_global_convergence}, after passing to a subsequence and discarding connected components that don't contain the origin, $\hat{\K}^j\to \hat{\K}$ smoothly and globally, where the limit $\hat{\K}$ is a generalized $(\al,\bar{\de})$-flow with convex -- and thus in particular connected -- time slices.

If $\hat{\K}$ doesn't contain surgeries, then it is a $\beta$-uniformly $2$-convex ancient $\alpha$-Andrews flow.
Otherwise, let $T\in(-\infty,0]$ be a surgery time and let $\hat{N}\subset \hat{K}_T^-$ be a surgery neck of quality $\bar{\de}$ sitting in the backward time slice.

\begin{claim}\label{claim_two_unbounded}
$\hat{K}_T^-\setminus \hat{N}$ has two unbounded components.
\end{claim}

\begin{proof}[Proof of Claim \ref{claim_two_unbounded}]
Note that $\hat{N}$ is the limit of some solid $\bar{\delta}$-necks $\hat{N}^j$ in the approximators $\hat{\K}^j$.
By part (2) of Definition \ref{def_MCF_surgery}, we can find a curve $\gamma_j$ in the approximator connecting
$\{H=\hat{H}_{\textrm{trig}}\}$ and $\{H\leq \hat{H}_{\textrm{th}}\}$, such that it passes through $\hat{N}^j$ but avoids all other $\bar{\delta}$-necks of the disjoint collection. We can assume that the curve $\gamma_j$ enters and leaves $\hat{N}^j$ exactly once
(if this isn't already the case, we can look at the earliest entry point and the latest exit point and change $\gamma_j$ for intermediate times to a curve within $\hat{N}^j$).
Note furthermore that $\gamma_j$ must intersect each of the two boundary discs of the cylinder exactly once (since otherwise we could modify it within $\hat{N}^j$ into a curve avoiding all necks completely, contradicting the minimal separation property).
Let $x_j$ be the center of $\hat{N}^j$. By the global curvature estimate (Theorem \ref{thm_glob_curv_est}), for every $\Lambda<\infty$ we have
\begin{equation}\label{ratio1}
{\hat{H}(x)}/{\hat{H}_{\textrm{neck}}}\leq C(\Lambda)<\infty\qquad\textrm{whenever}\,\, d(x,x_j)\leq \Lambda \hat{H}_{\textrm{neck}}^{-1}.
\end{equation}
By the Andrews condition and the local curvature estimate (Theorem \ref{loccurvest}), we also have a lower bound
\begin{equation}\label{ratio2}
{\hat{H}(x)}/{\hat{H}_{\textrm{neck}}}\geq c(\Lambda)>0\qquad\textrm{whenever}\,\, d(x,x_j)\leq \Lambda \hat{H}_{\textrm{neck}}^{-1}.
\end{equation}
Since ${H^j_{\textrm{trig}}}/{H^j_{\textrm{neck}}},{H^j_{\textrm{neck}}}/{H^j_{\textrm{th}}}\ra \infty$, given any $\Lambda<\infty$, for $j$ large enough the curve $\gamma_j$ must start and end outside $B(x_j, \Lambda \hat{H}_{\textrm{neck}}^{-1})$. Thus, $\hat{K}_T^-\setminus \hat{N}$ has at least two unbounded components. Since $\hat{K}_T^-$ is connected, $\hat{K}_T^-\setminus \hat{N}$ must have exactly two components.
\end{proof}

Since $\hat{K}_T^-$ has two ends (see Claim \ref{claim_two_unbounded}), it contains a line, and by monotonicity of sets all prior time slices contain this line also.
At each fixed time the convex set splits off an $\R$-factor, and thus there cannot be any other surgeries. It follows that $\hat{\K}$ is a round cylindrical flow for $t<T$ (c.f. the proof of Lemma \ref{lemma_limitneck}). Similarly, by the uniqueness of the standard solution (see Proposition \ref{prop_std_sol}), $\hat{\K}$ must be the evolution of the standard cap for $t>T$.

Finally, arguing as in \cite[Proof of Cor. 2.15]{HK} it follows that potential other connected components are cleared out, i.e. $\hat{\K}^j\to \hat{\K}$ smoothly and globally without the need of discarding connected components that don't contain the origin.
\end{proof}

\subsection{Existence of {$(\alpha,\delta,\mathbb{H})$-flows}}\label{subsec_existence}

We can now prove our main existence theorem for mean curvature flow with surgery, Theorem \ref{thm_main_existence}.

\noindent\emph{Outline of the proof.} We assume towards a contradiction that we have a sequence $\K^j$ of flows with degenerating $\mathbb{H}$-parameters that can be defined only on some finite maximal time intervals $[0,T_j]$.
For $j$ large, to obtain a contradiction, we want to argue that we can perform surgery and thus continue the flow beyond $T_j$. This amounts to finding suitable collection of $\de$-necks.
To this end, we first prove Claim \ref{claim_sepclaim} which shows that the thick and the trigger part in $K^j_{T_j}$ can be separated by a union of balls centered at boundary points with $H(p)=H_{\textrm{neck}}$ and radius comparable to $H_{\textrm{neck}}^{-1}$.
We then consider a minimal collection of balls with the separation property and prove that their centers are actually centers of strong $\hat{\delta}$-necks for any $\hat{\delta}$, see Claim \ref{claim_sepnecks}.
It is then easy to conclude the proof.

\begin{proof}[Proof of Theorem \ref{thm_main_existence}]
Let $\bar{\delta}=\bar{\delta}(\Balpha)>0$ small enough such that all previous estimates and also the argument in the last line of the present proof apply. 
By the maximum principle, the $\alpha$-Andrews condition and $\beta$-uniform $2$-convexity are preserved along smooth mean curvature flow \cite{andrews1,Huisken_convex}.
We assume towards a contradiction that for some $\delta\leq\ol{\delta}$ there is no constant $\Theta(\delta)<\infty$ such that the assertion of the theorem holds. Then there is a sequence $\K^j$ of $(\Balpha,\delta,\mathbb{H}_j)$-flows with ${H^j_{\textrm{trig}}}/{H^j_{\textrm{neck}}},{H^j_{\textrm{neck}}}/{H^j_{\textrm{th}}},H^j_{\textrm{neck}}\to\infty$, that can only be defined on a finite maximal time interval $I_j$.
If some $I_j$ were a half-open interval $[0,T_j)$, then the fact that $H\leq H^j_{\textrm{trig}}$ 
would allow us to pass to a limit as
$t\ra T_j$, so that $[0,T_j)$ is not maximal; therefore $I_j=[0,T_j]$ for some
$T_j<\infty$.
Moreover, it must be the case that we cannot find a minimal collection
of strong $\de$-necks in $K^j_{T_j}$
as required by the definition of an $(\Balpha,\delta,\mathbb{H}_j)$-flow (Definition \ref{def_MCF_surgery}), since otherwise we could cut along 
them (Proposition \ref{lemma_glue_caps}) and run smooth MCF for a short time, contradicting the maximality of $T_j$.\footnote{In particular, if $\{H(\cdot,T_j)\leq H^j_{\textrm{th}}\}=\emptyset$ everything is discarded at $t=T_j$, and thus the flow can be continued forever as empty flow, contradicting maximality.}
Therefore our goal is to  
produce such a collection for large $j$, to obtain a contradiction.

Let $\mathcal{I}_j$ be the set of points $p\in \D K^j_{T_j}$ with $H(p)>H^j_{\textrm{neck}}$, and let $\mathcal{J}_j$ be the set of points $p\in \D K^j_{T_j}$ with $H(p)=H^j_{\textrm{neck}}$.

\begin{claim}[Separation property]\label{claim_sepclaim}
There is a constant $C=C(\Balpha)\in(2N,\infty)$, such that the union $
V_j=\bigcup_{p\in \mathcal{J}_j} B(p,CH^{-1}(p))$, for $j$ large enough, separates $\{H=H^j_{\textrm{trig}}\}$ from $\{H\leq H^j_{\textrm{th}}\}$ in the domain $K^j_{T_j}$.
\end{claim}

\begin{proof}[Proof of Claim \ref{claim_sepclaim}]
We suppress $j$ in the notation. Fix $C\in(2N,\infty)$ to be determined later, and let $
U=\bigcup_{p\in \mathcal{I}} B(p,2NH^{-1}(p))$.
To establish the claim, it suffices to prove that $U\setminus V$ is open and closed in $K_T\setminus V$.
Indeed, for $j$ large enough we have (c.f. \eqref{ratio1} and \eqref{ratio2}) that $\{H=H_{\textrm{trig}}\}\subseteq U\setminus V$ and 
$\{H\leq H_{\textrm{th}}\}\subseteq K_T\setminus (U\cup V)$. Thus, once we know that $U\setminus V\subseteq K_T\setminus V$ is open and closed, it follows that the sets $\{H=H_{\textrm{trig}}\}$ and $\{H\leq H_{\textrm{th}}\}$ lie in two different components of $K_T\setminus V$.

Starting with the obvious part, the set $U$ is open since it is a union of open sets. Thus, $U\setminus V=U\cap (K_T\setminus V)$ is open in $K_T\setminus V$. 

Suppose now $x\in \ol{U\setminus V}$ lies in the closure of $U\setminus V\subseteq K_T\setminus V$.
We want to show that $x\in {U\setminus V}$.
There are sequences
$\{x_k\}\subset U$, $\{p_k\}\subset \{H>H_{\textrm{neck}}\}$ such that
$x_k\ra x$ and $x_k\in B(p_k,2NH^{-1}(p_k))$.  Passing to subsequences,
we may assume that $p_k\ra p$.  Then $H(p)\geq H_{\textrm{neck}}$.  If
$H(p)=H_{\textrm{neck}}$, then $x\in B(p,CH^{-1}(p))\subseteq V$
since $C>2N$; a contradiction to the assumption that $x\in\ol{U\setminus V}\subseteq K_T\setminus V$. Hence we have
$H(p)>H_{\textrm{neck}}$.
Let $y\in \D K_T$ be a point nearest $x$.
Then $d(y,x)\leq (N-1)H^{-1}(y)$,  in particular $x\in B(y,2NH^{-1}(y))$. If $H(y)> H_{\textrm{neck}}$, we obtain that $x\in U$ and thus that $x\in U\setminus V$, what we wanted to show.
If $H(y)=H_{\textrm{neck}}$,  we obtain $x\in V$; a contradiction.

Finally, let us rule out the remaining case $H(y)<H_{\textrm{neck}}$. Since $y$ is a boundary point nearest to $x$, we have $d(y,p)\leq 2d(x,p)\leq 4N H^{-1}(p)$.
Thus, by Proposition \ref{lem_quasi_convex}, Proposition \ref{prop_std_sol} and Theorem \ref{prop_selfimprove} there is a $C_1=C_1(\Balpha)<\infty$
such that, for $j$ large enough,  $y$ and $p$ lie in the same
connected component of $B(p,C_1H^{-1}(p))\cap \D K_T$.
Since  $H(y)<H_{\textrm{neck}}<H(p)$, there is a $z\in B(p,C_1H^{-1}(p))\cap \D K_T$ with $H(z)=H_{\textrm{neck}}$. Note that $d(z,x)\leq d(z,p)+d(p,x)\leq (C_1+2N) H^{-1}(p)$, and thus $H(z)d(z,x)\leq (C_1+2N)\tfrac{H(z)}{H(p)}\leq C$, for some $C=C(\Balpha)<\infty$ by Theorem \ref{thm_glob_curv_est}.
This shows $x\in B(z,CH^{-1}(z))\subseteq V$; a contradiction.
\end{proof}

Let $\hat{\mathcal{J}}_j\subseteq\mathcal{J}_j$ be a minimal (and thus finite) subset
such that the union of balls
$\cup_{p\in \hat{\mathcal{J}}_j}B(p,CH^{-1}(p))$ has the separation property.
%We remark that $\mathcal{J}_j$ can of course contain some points that do not look neck-like. However, thanks to the minimality of $\hat{\mathcal{J}}_j$, we have:

\begin{claim}[Strong neck point property]\label{claim_sepnecks}
Given any $\hat\de>0$, for $j$ large enough all points in $\hat{\mathcal{J}}_j$ are strong $\hat{\delta}$-neck points.
\end{claim}

\begin{proof}[Proof of Claim \ref{claim_sepnecks}]
Suppose that for all $j$ one of the points $p_j\in \mathcal{J}_j$
is not a strong $\hat\de$-neck.  By Theorem \ref{prop_selfimprove}, parabolically rescaling to make $H(p_j,T_j)=1$ and passing to a
subsequence $\hat{\K}^j$, we get a limit ancient $(\al,\de)$-flow $\hat{\K}$ with basepoint 
$(p,t)=(0,0)\in \D \hat{\K}$ that is either (a) a $\beta$-uniformly $2$-convex ancient $\al$-solution, or (b) an evolution of the standard surgery cap preceded by a round shrinking cylinder.
In case (a), since the balls separate and the curvature ratios tend to infinity, the limit $\hat{\K}$ is clearly non-compact.
Consider a $\hat{\delta}$-neck $\hat{N}$ at a point $q$ at $t=0$ provided by Proposition \ref{lem_quantitative_one_ended}. Here we chose $q$ far enough away to ensure that $\hat{N}$ and $B(p,2C)$ are disjoint.
This neck $\hat{N}$ is the limit of some $\hat{\delta}$-necks $\hat{N}^j$ in the approximators.
As in the proof of Claim \ref{claim_two_unbounded}, let $\gamma_j$ be a curve in the approximators that connects the trigger and the thick part and passes through $B(p_j,C)$, but avoids all other balls from the minimal collection of separating balls.
Always assuming $j$ is large enough, we can short circuit $\gamma_j$ inside $\hat{N}^j$.
After this modification, $\gamma_j$ misses $B(p_j,C)$. It also still misses all other balls from our minimal collection, unless $\hat{N}^j$ itself intersects a ball $B'=B(p_j',C)$ for some $p_j'\in \mathcal{J}_j\setminus \{p_j\}$.    This however is impossible, because in this case one of the complementary components would have uniformly bounded diameter.
Thus, we get a contradiction with the separation property of $\hat{\mathcal{J}}_j$.
Finally, in case (b), by Proposition
\ref{prop_std_sol} we get a contradiction similar to the one in case (a).
\end{proof}

By Claim \ref{claim_sepnecks}, for large $j$, every $p\in \hat{\mathcal{J}}_j$ is a strong $\de$-neck point.  These necks are disjoint for large $j$, since otherwise again by Claim \ref{claim_sepnecks} two intersecting $\de$-necks would lie in a single $\hat\de$-neck for $\hat{\de}\ll\de$, which is impossible by minimality of $\mathcal{J}_j$.
Thus, we have a minimal collection of disjoint strong $\delta$-necks with the separation property (since $C(\Balpha)>2N$ and $\bar{\delta}$ is small enough); this gives the desired contradiction.
\end{proof}

\subsection{Further properties}\label{sec_properties}
Finally, for convenience of the reader we explain how Corollaries \ref{cor_discarded}, \ref{cor_topo} and Proposition \ref{cor_levelset} are obtained as easy consequences
(in fact, the proof of Proposition \ref{cor_levelset} could have been given right after stating the axioms of an $(\Balpha,\de,\mathbb{H})$-flow).

\begin{proof}[Proof of Corollary \ref{cor_discarded}]
Fix $\eps_1=\eps_1(N)\ll 1$ such that the gluing argument below works (Claim \ref{claim_gluing}).
Fix $\bar{\eps}\ll \min\{\eps_1,\underline{R}_1^{-1},\underline{R}_2^{-1}\}$,
where $\underline{R}_i=\underline{R}(\eps_1/2,\eps_1/2,\alpha,\beta)$ are the constants from
Proposition \ref{structure_compact} and Proposition \ref{lem_quantitative_one_ended}, respectively, and let $\underline{R}_3=\underline{R}_3(\bar{\eps},\tau(\bar{\eps}))$ be the constant from Proposition \ref{prop_std_sol}.
We will prove that the corollary holds for any $\eps\ll \min\{\bar{\eps},\underline{R}_3^{-1}\}$.

Let $C$ be a discarded component, and consider the (possibly empty) set $\mathcal{I}\subseteq \D C$ of $\eps_1$-neck points.

\begin{claim}\label{claim_gluing}
There is a set $N$, $\mathcal{I}\subseteq N\subseteq C$, such that for each $q\in N$ there is a point $p\in\mathcal{I}$ with $d(p,q)\leq 100 H^{-1}(p)$, and such that
either $N=C\cong \bar{D}^{N-1}\times S^1$  or $N\subsetneq C$ and each component of $N$ has the topology of $\bar{D}^{N-1}\times I$, for some open interval $I$.
\end{claim}

\begin{proof}[Proof of Claim \ref{claim_gluing}]
Let $\mathcal{I}'\subseteq \D C$ be a maximal collection of $\eps_1$-neck points such that for any pair $p,q\in \mathcal{I}'$ the separation between them is at least $50\min\{H^{-1}(p),H^{-1}(q)\}$.
Note that each $p\in \mathcal{I}'$, being an $\eps_1$-neck point for $\eps_1$ small, comes with a cylindrical neighborhood of length much longer than $50H^{-1}(p)$.
Following this cylinder in each direction, we see that there are $0$, $1$ or $2$ neighboring points $p'\in \mathcal{I}'$, where by neighboring point we mean a point of $\mathcal{I}'$ closest to $p$ in a direction of the cylinder and at distance at most $110 H^{-1}(p)$.
Let $C_p$ be the open core of a neck at $p$ going say $75\%$ of the distance towards each neighboring point (and going say $10H^{-1}(p)$ in directions without neighboring point), and let $N=\cup_{p\in\mathcal{I}'}C_p$. 
By construction, the set $N$ is built by gluing together necks with intersection multiplicities at most $2$ and substantial overlap. Since $\eps_1$ is small enough, it follows that the connected components of $N$ must be either $\bar{D}^{N-1}\times S^1$ or $\bar{D}^{N-1}\times I$. Since $C$ is connected, this implies the claim.
\end{proof}

If $N=C$, then $C$ has the topology of $\bar{D}^{N-1}\times S^1$, and we are done.
If $N=\emptyset$ then, always assuming $\eps$ is sufficiently small,
it follows from Theorem \ref{thm_can_nbd}, Proposition \ref{structure_compact} and Proposition \ref{prop_std_sol}, that $C$ is modeled on a compact ancient $\alpha$-solution of controlled geometry, and thus that is has the topology of $\bar{D}^N$.
Assume now $\emptyset\subsetneq N\subsetneq C$, and let $N'\subseteq N$ be a component of maximal diameter.

We first consider the case $\diam N'\leq 10\underline{R} H^{-1}(p)$ for some $\eps_1$-neck point $p\in \mathcal{I}\cap N'$,
where $\underline{R}=\max\{\underline{R}_1,\underline{R}_2\}$.
If the $\eps$-model at $p$ provided by Theorem \ref{thm_can_nbd} were the cylinder, then the connected component of $\mathcal{I}$ containing $p$ would have diameter larger than $\tfrac{1}{\eps} H^{-1}(p)$, contradicting the assumption that
 $\diam N'\ll \tfrac{1}{\eps}H^{-1}(p)$ and $\mathcal{I}\subseteq N$.
If the $\eps$-model at $p$ were a noncompact $\alpha$-solution or a compact $\alpha$-solution with diameter much larger than $\underline{R} H^{-1}(p)$,
then we could apply Proposition \ref{lem_quantitative_one_ended} (see also Remark \ref{remark_one_ended_cpt}) at a point $p'\in \D N'\cap \D C$,
which by definition is not in $\mathcal{I}$ and has curvature comparable to $p$ by Theorem \ref{thm_glob_curv_est},
and would again get, c.f. Claim \ref{claim_gluing}, a much longer $\eps_1$-neck at controlled distance, contradicting the assumption that $N'$ is maximal.
If the $\eps$-model at $p$ were a standard solution, Proposition \ref{prop_std_sol} would give a similar contradiction.
Thus, the $\eps$-model at $p$ must be a compact $\alpha$-solution of controlled size, and thus $C$ has the topology of $\bar{D}^N$.

Let us now consider the remaining case that $\diam N'> 10\underline{R} H^{-1}(p)$ for all $\eps_1$-neck points $p\in \mathcal{I}\cap N'$.
Since for each $q\in N$ there is a point $p\in\mathcal{I}$ with $d(p,q)\leq 100 H^{-1}(p)$ and since $\eps_1$ is small, it follows that
\begin{equation}\label{eqn_diam_est}
\diam N'> 9\underline{R} H^{-1}(p)
\end{equation}
actually holds for all points $p\in N'\cap \D C$.
Select points $p_\pm\in \D N'\cap \D C$ on the boundary circles of the neck $N'$. By definition we have $p_\pm\notin \mathcal{I}$.
We now apply Theorem \ref{thm_can_nbd}, Proposition \ref{lem_quantitative_one_ended} and Proposition \ref{prop_std_sol}
with center $p_\pm$.
If we are not in case (3) of Proposition \ref{prop_std_sol}, then we get caps $C_\pm$ with cylindrical collars of length say $R=2\underline{R}$.
By \eqref{eqn_diam_est} the caps are disjoint, and since $R=2\underline{R}$ their collars intersect $N'$ with substantial overlap. Thus, the conclusion is that $C$ has the topology of a ball $\bar{D}^{N}$.
Finally, if we are in case (3) of Proposition \ref{prop_std_sol}, then we see that actually $\diam  N'\gg \max\{\underline{R}_1,\underline{R}_2,\underline{R_3}\}H^{-1}(p)$ and thus the above argument applies.
\end{proof}

\begin{proof}[Proof of Corollary \ref{cor_topo}]
Let $K_0$ be a smooth compact $2$-convex domain in $\R^N$. By Remark \ref{remark_controlled} it is $\Balpha$-controlled for some $\Balpha$.
We choose constants as in Remark \ref{remark_practice}. By Theorem \ref{thm_main_existence} there exists an $(\Balpha,\delta,\mathbb{H})$-flow $\K=\{K_t\}_{t\in [0,\infty)}$ starting at $K_0$.
Since $\K$ becomes extinct in finite time (Remark \ref{remark_extinct}), by Corollary \ref{cor_discarded} it provides a decomposition of $K_0$ into a finite connected sum of solid tori $\bar{D}^{N-1}\times S^1$, as claimed.
\end{proof}

\begin{proof}[Proof of Proposition \ref{cor_levelset}] It suffices to prove the following claim:
\begin{claim}[c.f. {\cite[Prop. 2.2]{Lauer}}]
For every $r>0$ there is a $j_0<\infty$, such that if $j\geq j_0$, $t$ is a surgery time of $\K^j$, and $B(p,r)\subseteq K_t^{j,-}$, then $B(p,r)\subseteq K_t^{j,+}$.
\end{claim}
\begin{proof}
Since $B(p,r)\subseteq K_t^{j,-}$ clearly doesn't fit into a very thin and long neck, c.f. Definition \ref{def_surgery}, we must have $B(p,r)\subseteq K_t^{j,\sharp}$, provided $\bar{\de}<<1$ and $j$ is sufficiently large.
If $B(p,r)$ was contained in a discarded component, then we could find a boundary point with $H\leq (N-1)r^{-1}$. However, by Definition \ref{def_MCF_surgery} the discarded components have $H>H^j_{\textrm{th}}$ everywhere. Thus, $B(p,r)\subseteq K_t^{j,+}$ for $j$ large enough.
\end{proof}
As explained in \cite{Lauer}, the corollary now follows easily.
\end{proof}

\begin{remark}
The same argument applies to more general $(\al,\de)$-flows $\K^j=\{K^j_t\subset\R^N\}_{t\in[0,\infty)}$ with fixed $\Balpha$-controlled initial condition $K_0$,
assuming only that $s_\sharp(\K^j)\to 0$, $\delta\leq\bar{\de}\ll 1$, and that the minimum of the mean curvature of the discarded components goes to infinity.
\end{remark}

\begin{appendix}

\section{One-sided minimization}\label{appendix_min}

\begin{proof}[Detailed proof of Proposition {\ref{lemma_onesidedmin}}]
For convenience of the reader, we now explain the details of the argument of White \cite[Sec. 3]{white_size} and Head \cite[Sec. 5]{Head}.
Let $K_{t_0},K_{t_1}$ and $\bar{B}$ be as in the statement of the proposition. If $t_1$ is a surgery time, we first consider the case that $K_{t_1}$ is interpreted as $K_{t_1}^-$. By standard weak compactness and weak lower semicontinuity of perimeter, there is a set $X\subset \R^N$ that minimizes
the perimeter (in the following $\partial$ stands for the reduced boundary)
\begin{equation}
 \textrm{Per}_{\bar{B}}(X)=\abs{\D X\cap \bar{B}},
\end{equation}
among all sets of locally finite perimeter satisfying
\begin{equation}
K_{t_1}\subseteq X\subseteq K_{t_0} \qquad\textrm{and}\qquad
X\setminus \bar{B}=K_{t_1}\setminus \bar{B}.
\end{equation}
We consider the support of $X$, which is the closed set
\begin{equation}
\spt(X)=\{x\in \R^N\mid |B(x,r)\cap X|>0\;\text{for all}\;r>0\}\,.
\end{equation}
We have to show that $\spt(X)\subseteq K_{t_1}$.
Let $\bar{t}\in[t_0,t_1]$ be the supremum of the times $t\in [t_0,t_1]$
such that $\spt(X)\subseteq K_{t}^-$. Since $\cap_{t<\bar{t}}K_{t}^-=K_{\bar{t}}^-$,
we have $\spt(X)\subseteq K_{\bar{t}}^-$. Thus we are done if $\bar{t} = t_1$,
so we assume towards a contradiction that $\bar{t}<t_1$.

The idea is now to consider suitable mean convex domains $D\subset \R^N$ with $\spt(X)\subseteq D$ (or at least such that $\spt(X)\cap B(x,r)\subseteq D\cap B(x,r)$) and to obtain a contradiction
 with the `maximum principle' at any point $x\in\spt(X)$ where $\D D$ and $\D X$ touch.
To make this rigorous we use the following standard fact from geometric measure theory:
\begin{lemma}
\label{lem_support}  Suppose
$D\subset \R^N$ is a domain with smooth boundary, $x\in \D D\setminus K_{t_1}^-$ 
is a boundary
point at which  the mean curvature of $\D D$ is strictly positive, and that for some 
$r>0$ we have the inclusions
$$
\spt(X)\cap B(x,r)\subseteq D\cap B(x,r)\subseteq \bar{B}\cap B(x,r)\,.
$$
Then $x\not\in\spt(X)$. 
\end{lemma}
\begin{proof}
If $x\in \spt(X)$, then since the open set $B(x,r)\setminus D$ lies in 
$\R^N\setminus \spt(X)$, it follows that $x$ belongs to the support of the
perimeter measure $\textrm{Per}_{\bar{B}}(X)$.  Using a slight deformation of $\D D$, one can 
construct a foliation of a neighborhood of $x$ by mean convex hypersurfaces,
such that one of the leaves has strict one-sided exterior contact with $D$ at
$x$.  
Using a strictly area
decreasing deformation one gets a contradiction to the fact
that $X$ is minimizing.  
\end{proof} 

For every $x\in \D \bar{B}\setminus K_{t_1}^-$, we may apply Lemma 
\ref{lem_support} with 
$D=\bar{B}$ to conclude that $x\not\in\spt(X)$. Similarly, for
every $x\in \D K_{\bar t}^-\setminus\D K_{t_1}^-$, we may apply the lemma with $D=K_{\bar{t}}^-$, to conclude that $x\not\in\spt(X)$. If $\bar{t}$ is not a surgery time, this contradicts the maximality of $\bar{t}<t_1$.

If $\bar t$ is a surgery time, consider the post-surgery domain $K_{\bar t}^\sharp\subseteq K_{\bar t}^-$ (see Definition \ref{def_alphadelta}).
Suppose $p$ is the center of a surgery neck at time $\bar t$ and scale $s$ such that
$B(p,5\Gamma s)\cap (K_{\bar t}^-\setminus K_{\bar t}^\sharp)\cap\bar{B}\neq \emptyset$.
Suppose $\spt(X)\cap B(p,10s)\neq\emptyset$.
Since 
$\spt(X)\subseteq K_{\bar t}^-$, when $\Gamma$ is sufficiently large, 
we may find a mean convex surface (e.g. a slight deformation of the catenoid) that has $1$-sided contact with
$\spt(X)$ at some point 
$x\in B(p,\Gamma s)\cap (K_{\bar t}^-\setminus K_{\bar t}^\sharp)$.
By assumption we have $x\in \bar{B}$, and as argued above, we have
$x\in B$.  Then Lemma \ref{lem_support} gives a contradiction.  Therefore
$\spt(X)\cap B(p,10s)=\emptyset$.
By one-sided contact with halfspaces (Lemma \ref{lem_support}) and almost convexity of the caps, we get that $K_{\bar t}^\sharp$ must be contained in thin neighborhoods of the caps.
Since these thin neighborhoods are foliated by mean convex sets, namely the level sets of the distance function, we conclude that $\spt(X)\subseteq K_{\bar t}^\sharp$, again by Lemma \ref{lem_support}.  It follows
from the minimality of $X$ that 
 $\spt(X)\cap C=\emptyset$ for any connected component $C$ of $K_{\bar t}^\sharp$
which is thrown away to form $K_{\bar t}^+$.  Thus $\spt(X)\subseteq K_{\bar t}^+$.
Finally, since $\bar t<t_1$, by Lemma \ref{lem_support}
we get that $\spt(X)\cap \D K_{\bar t}^+=\emptyset$.
Since $K_t\ra K_{\bar t}^+$ as $t\ra \bar t+$, 
it follows   that $\spt(X)\subseteq K_t^-$ for some $t>\bar t$.  This
contradicts the definition of $\bar t$, and 
we are done.

A similar argument applies if $K_{t_1}$ is interpreted as $K_{t_1}^\sharp$ or $K_{t_1}^+$, except that
in the case that $\bar{t}=t_1$, one repeats the above argument once more to see that
$\spt(X)\subseteq K_{{t_1}}^\sharp$ respectively $\spt(X)\subseteq K_{{t_1}}^+$.
\end{proof}

\begin{remark}[Elementary argument] We now sketch an alternative and more elementary argument.
By \cite[Rmk 2.6]{HK} it is enough to construct a vector field $X$ with $\textrm{div} X\geq 0$, $\abs{X}\leq 1$ and $X$ equal to the outward unit normal on $\partial  K_{t}$.
In the case without surgeries and discarded components, such a vector field is given by the outward unit normal of the mean convex foliation.
In our case, the vector field given by the outward unit normal might not be defined everywhere, so we have to extend it.
We can take care of the discarded pieces by defining $X$ as the negative gradient of the distance to the boundary.\footnote{Since the boundary is mean convex, $\textrm{div} X\geq 0$ in the barrier sense, and thus also in the distributional sense, which is good enough for the calibration argument.}
We can take care of the surgeries, by interpolating between the vector field of the cylinder and the vector field of the cap, using a slowly varying cutoff function along the axis of the cylinder. Due to the formula
\begin{multline}
 \textrm{div}(\varphi_1 X_1+\varphi_2 X_2)\\
=\varphi_1\textrm{div}(X_1)+\varphi_2\textrm{div}(X_2)+\langle\nabla\varphi_1,X_1\rangle+\langle\nabla\varphi_2,X_2\rangle,
\end{multline}
we can do this preserving positive divergence, provided the cap separation parameter $\Gamma$ is large enough.
\end{remark}

%\section{Convex slab stuff}

%\begin{lemma}
%Suppose $C\subset \R^N$ is a closed convex set with nonempty interior.
%Then the boundary $\D C$ is connected unless $C$ is a slab, i.e.
%a product of a hyperplane
%with an interval.
%\end{lemma}
%\begin{proof}
%Assume $\D C$ is disconnected.
%Then $C$ is noncompact, since otherwise $\D C\simeq S^{N-1}$.   Suppose $\ga\subset C$ is a line segment joining distinct connected components of $\D C$.
%Without loss of generality we may assume that $0$ is an interior point of $\ga$.
%Then $C$ is star-shaped, and the set
%$W=\{v\in S^{N-1}\mid [0,\infty)v\subset C\}$ is a closed $\pi$-convex set, and the complement $W^c=S^{N-1}\setminus W$ is disconnected.

%Now suppose that $N=2$.  Then $W=S^1\setminus W$ is a closed, $\pi$-convex set with disconnected complement $W$, and must therefore consist of a single pair
%of antipodal points.

%Returning to general case, by considering planes
%$P\subset \R^N$ which contain $\ga$, and applying the $N=2$ case, obtain
%$N-2$ pairs $\{(v_1^\pm,\ldots,v_{N-2}^\pm\}\subset W$
%such that $v_i^-=-v_i^+$, whose linear span is a hyperplane $H\subset \R^N$.
%Then $H\subset C$, and $C$ is a union of hyperplanes parallel to $H$.
%Since $C$ has more than one boundary component it must be a slab
%$H\times I$.
%\end{proof}

\end{appendix}

\bibliography{surgery}

\bibliographystyle{alpha}

\vspace{10mm}
{\sc Courant Institute of Mathematical Sciences, New York University, 251 Mercer Street, New York, NY 10012, USA}

\emph{E-mail:} robert.haslhofer@cims.nyu.edu, bkleiner@cims.nyu.edu

\end{document}